\documentclass[12pt,oneside]{amsart}
\usepackage{lineno}
\usepackage{graphics}
\usepackage{multicol}
\usepackage{enumerate}
\usepackage[dvips]{graphicx}
\usepackage{cancel}
\usepackage{amsmath,amsfonts}
\usepackage[update,prepend]{epstopdf}
\usepackage{color}
\usepackage{amssymb}
\usepackage{latexsym}
\usepackage{epsfig}
\usepackage{geometry}
\usepackage[update,prepend]{epstopdf}
\usepackage{hyperref}
\usepackage{setspace}
\usepackage[normalem]{ulem}

\newtheorem*{defin*}{Definition}
\newtheorem{theo}{Theorem}[section]
\newtheorem{prop}{Proposition}[section]

\newtheorem{lem}{Lemma}[section]

\newtheorem{cor}{Corollary}[section]
\newtheorem{defin}{Definition}[section]
\newtheorem{rem}{Remark}[section]
\newtheorem{assum}{Assumption} [section]

\def\Re{\operatorname{{Re}}}
\def\Im{\operatorname{{Im}}}

\def\ld{{L^2(\Om)}}

\def\ga{{\gamma}}

\def\Ga{{\Gamma}}

\def\Om{{\Omega}}
\def\al{{\alpha}}

\def\la{{\lambda}}

\def\Re{\operatorname{{Re}}}
\def\Im{\operatorname{{Im}}}

\newcommand{\ds}{\displaystyle}

\def\be{\begin{equation}}
\def\ee{\end{equation}}
\def\bq{\begin{eqnarray}}
\def\eq{\end{eqnarray}}
\def\bqs{\begin{eqnarray*}}
	\def\eqs{\end{eqnarray*}}
\def\wa{\begin{aligned}}
	\def\ali{\end{aligned}}
\def\ben{\begin{enumerate}}
	\def\enu{\end{enumerate}}
\def\bmu{\begin{multicols}}
	\def\emu{\end{multicols}}
\def\dis{\displaystyle}
\usepackage[noadjust]{cite}
\usepackage{filecontents}
\usepackage{hyphenat}
\usepackage[utf8]{inputenc}
	\author[Corr\^ea]{Wellington Jos\'e Corr\^ea}
	\address{Academic Department of Mathematics, Federal Technological University of  Paraná, Campuses Campo Mour\~{a}o,  87301-899, Campo Mour\~{a}o, PR, Brazil
	}
	\email{wcorrea@utfpr.edu.br}
	\author[\''{O}zsar\i]{T\"{u}rker \"{O}zsar\i\textsuperscript{†}}
	\address{Department of Mathematics,
		Izmir Institute of Technology,
		Izmir, Turkey}
	\email{turkerozsari@iyte.edu.tr}
\thanks{\textsuperscript{†}Türker Özsarı's research was supported by TÜBİTAK 3501 Career Grant \#115F055}
 \title[Ginzburg-Landau equations with dynamic boundary conditions]{Complex Ginzburg-Landau equations with dynamic boundary conditions}
 
\begin{document}
 	\maketitle
 	\begin{abstract}
 		The initial-dynamic boundary value problem (idbvp) for the complex Ginzburg-Landau equation (CGLE) on bounded domains of $\mathbb{R}^N$ is studied by converting the given mathematical model into a Wentzell initial-boundary value problem (ibvp). First, the corresponding linear homogeneous idbvp is considered.  Secondly, the forced linear idbvp with both interior and boundary forcings is studied. Then, the nonlinear idbvp with Lipschitz nonlinearity in the interior and monotone nonlinearity on the boundary is analyzed.  The local well-posedness of the idbvp for the CGLE with power type nonlinearities is obtained via a contraction mapping argument.  Global well-posedness for strong solutions is shown.  Global existence and uniqueness of weak solutions are proven.  Smoothing effect of the corresponding evolution operator is proved. This helps to get better well-posedness results than the known results on idbvp for nonlinear Schrödinger equations (NLS).  An interesting result of this paper is proving that solutions of NLS subject to dynamic boundary conditions can be obtained as inviscid limits of the solutions of the CGLE subject to same type of boundary conditions. Finally, long time behaviour of solutions is characterized and exponential decay rates are obtained at the energy level by using control theoretic tools.
 	\end{abstract}
\tableofcontents
 	\section{Introduction}
This article is devoted to the analysis of the initial-dynamic boundary value problem (idbvp) for the complex Ginzburg-Landau equation (CGLE):
\begin{equation}\label{GL_equation}
 	\begin{cases}
 	u_t-(\lambda+i\alpha)\triangle u+f(u)=0 &\text { in } \Omega\times \mathbb{R_+},\\
 	\displaystyle\frac{\partial u}{\partial \nu} = -g(u_t) &\text { on } \Gamma_1\times \mathbb{R_+},\\
 	u=0  &\text { on } \Gamma_0\times \mathbb{R_+},\\
 	u(0)=u_0  &\text { in } \Omega.
 	\end{cases}
\end{equation}

In \eqref{GL_equation}, $\Omega\subset\mathbb{R}^N$ is a bounded regular domain with boundary $\Gamma$, which is the union of $\Gamma_0$ and $\Gamma_1$, two nonempty, non-intersecting, connected $(n-1)$-dimensional manifolds. $u=u(x,t)$ is a complex valued function that denotes the complex oscillation amplitude.  $f(u)$ will be defined either as the usual power type nonlinearity $f(u)=(\kappa+i\beta)|u|^{p-1}u-\gamma u$ (Sections 6-10) or as an appropriate Lipschitz function (Section 5). Here,  $\beta\in\mathbb{R}$ and $\alpha\in\mathbb{R}$ are the (nonlinear) frequency and (linear) dispersion parameters, respectively. Without loss of generality, $\alpha$ can be taken as positive. Therefore we will also assume $\alpha>0$ throughout the text. $\beta$ can have both signs except when we discuss global solutions with power type nonlinearities, where it will be assumed to be positive. $\displaystyle\frac{\partial u}{\partial \nu}$ denotes the unit outward normal derivative.  $p\ge 2$ is the source power index. Other parameters satisfy $\lambda,\kappa>0,\gamma\in\mathbb{R}$.   $g:\mathbb{C}\rightarrow\mathbb{C}$ is a complex valued function which is either taken as identity (Sections 2-4,6-10) or as a monotone function (Section 5) satisfying suitable growth conditions to be specified later in Assumption \ref{Assumption_g}.

CGLE is a fundamental model in mathematical physics to describe near-critical instability waves, such as a reaction diffusion system near a Hopf-bifurcation.  Concrete applications of this equation include nonlinear waves, second-order phase transitions, superconductivity, superfluidity, Bose-Einstein condensation, and liquid crystals. See \cite{IL2002} and the references therein for an overview of several phenomena described by the CGLE.

CGLE simultaneously generalizes the nonlinear heat and nonlinear Schrödinger equations (NLS), both of which can be obtained in the limit as the parameter pairs $(\alpha,\beta)$ and $(\lambda,\kappa)$ tend to zero, respectively.  Therefore, it is natural to expect that CGLE carries some of the characteristics of the nonlinear heat equation and NLS.  The latter two types of equations have been studied to some extent under dynamic boundary conditions.  However, there has been no such progress for the CGLE.  Most models assumed ideal set-ups neglecting possible linear and nonlinear interior-boundary interactions.  See for example \cite{Bechouche}, \cite{cazenave}, \cite{Clement}, \cite{doering}, \cite{Ginibre1}, \cite{Ginibre2}, \cite{Nader}, \cite{okazawa},  \cite{okazawa2002monotonicity}, \cite{Shimo}, \cite{shimotsuma}, and \cite{tang} for existence and non-existence results on the CGLE in the case of the whole space or domains with homogeneous or periodic boundary conditions.    There are only a few results on the CGLE under nonhomogeneous boundary conditions (\cite{AamSmyKrs2005}, \cite{AamSmyKrs2007}, \cite{Gal_GL}-\cite{Gao_GL_1}, \cite{Rosier}).  The NLS subject to inhomogenenous or nonlinear boundary conditions, which can be considered a limiting case of the CGLE, took much more attention in recent years; see for example \cite{Audiard1}, \cite{Audiard2}, \cite{Batal}, \cite{erdogan}, \cite{fokas}, \cite{holmer}, \cite{elena}, \cite{Kalantarov}, \cite{Lasiecka1}, \cite{Lasiecka2}, \cite{Ozsari}-\cite{Ozsar2015}, and \cite{Strauss}.

Recently,  \cite{Dynamic_BC}  studied the defocusing cubic Schrödinger equation with dynamic boundary conditions on a bounded domain $ \Omega\subset\mathbb{R}^N$ with smooth boundary for $N=2,3$. The model considered in \cite{Dynamic_BC} was the special case of the problem \eqref{GL_equation} where $ \alpha=\beta=1$ and $ \lambda=\kappa=\gamma=0. $ In this work, the authors obtained the local well-posedness of strong ($H^2$) solutions for $N=2,3$ and global well-posedness of strong solutions for $N=2.$ In addition,  the existence (without uniqueness) of  weak ($H^1$) solutions was obtained for $N=2,3$. Moreover, it was proven that the energy of the weak solutions satisfies a uniform decay rate estimate under appropriate monotonicity  conditions imposed on  the nonlinear term appearing in  the dynamic boundary conditions.

The key idea in \cite{Dynamic_BC} is replacing the given dynamic boundary condition with an equivalent boundary condition, which is obtained by replacing $u_t$ on the boundary with the Laplacian and other terms coming from the main equation.  This enables one to obtain the generation of a semigroup in an appropriate topology.  The idea of using a boundary condition which involves the trace of the Laplacian comes from Venttsel's work \cite{Wentzell}. In his paper, Venttsel was interested in finding the most general boundary condition which restricts the closure of a given elliptic operator to the infinitesimal generator of a semigroup of positive contraction operators on the Banach space of continuous functions over a regular compact region \cite{Wentzell}.  The result of this work was the discovery of the generalized Venttsel (more commonly "Wentzell") boundary condition
$\displaystyle a\Delta u +b\frac{\partial u}{\partial \nu} + c u = 0 \text{ on }\Gamma$,
 	which provided the desired property for $a>0, b,c \geq 0$.
 	
Physically, this boundary condition can be considered as a (damped) harmonic oscillator acting at each point on the boundary. In the case of the heat equation, this means that the boundary can act as a heat source or sink depending on the physical situation. These boundary conditions also arise naturally in the study of the wave equation. In particular, generalized Wentzell boundary conditions can be thought of as a closed subclass of acoustic boundary conditions. The well-posedness of Wentzell problems for the heat equation was proved on spaces of the form $X_p = L^p(\Omega)\cup L^p(\Gamma)$. See for example \cite{Goldstein}, where the Wentzell problem is treated as a coupled system of two PDEs: one on the interior and one on the boundary. Regarding Wentzell boundary conditions,  it is  also worth mentioning the papers \cite{Dynamic_BC}-\cite{Cavalcanti1}, and \cite{Goldstein}-\cite{Goldstein2} on heat and wave equations.

We are interested in studying three main problems for the CGLE model considered in \eqref{GL_equation}: (i) well-posedness (ii) inviscid limits (iii) long-time behaviour.  However, there are some challenges:
\begin{enumerate}[(i)]
    \item The method used on the heat equation \cite{Goldstein} fails when $\alpha\neq 0.$
	\item The linear version of \eqref{GL_equation} does not generate a semigroup at the $L^2-$level.  Since $ L^2(\Om) $ is not the natural underlying topology for problem \eqref{GL_equation} for semigroup generation, one cannot directly employ the monotone operator theory \cite{okazawa2002monotonicity} (a common tool to treat CGLE) since the operator $ B\,u=|u|^p\,u $ is $ m $-accretive only in $ L^2(\Om), $ but not for example in $ H^1(\Om). $
	\item There is no control of the $ L^2(\Omega)-$norm of the solution due to the presence of a non-standard boundary condition. This is a major drawback for treating the nonlinear terms via Gagliardo-Nirenberg type estimates.  Focusing problems become particularly difficult even under smallness assumptions on the power of the nonlinearities or initial datum.
\end{enumerate}

One can overcome some of the difficulties above by using the method presented in \cite{Dynamic_BC}. However, the additional terms in \eqref{GL_equation} compared to the NLS equation make the analysis more subtle.

There are several differences and improvements in our work compared to that of \cite{Dynamic_BC}, who studied cubic defocusing NLS subject to dynamic boundary conditions:
\begin{enumerate}[(i)]
    \item We prove that solutions of the CGLE possess better interior regularity than the solutions of the NLS.  This result verifies the natural smoothing effect of the semigroup generator (see for example Theorem \ref{theo T}, Corollary \ref{corollary_inho}, Lemma \ref{H1global}, and Theorem \ref{main_density_theo}).  This latter property is due to the parabolic component of the evolution operator, which is a missing ingredient in the case of the Schrödinger equation.
	\item Our proof of the local existence of strong solutions is slightly different because the solution operator constructed in \cite{Dynamic_BC} is not necessarily a contraction for our model on any closed ball of the solution space.  Our approach requires the construction of a special complete metric space (Lemma \ref{QTspace}) whose elements are compatible at time $t=0$ with the initial datum $u_0$. We show that the solution operator maps this complete metric space (actaully a suitably chosen closed ball in it) onto itself in a contractive manner.  Due to our construction of this complete metric space, the initial datum satisfies the necessary compatibility condition at each step of the contraction argument, which allows us to use the linear non-homogenous theory.
\item Controlling the $H^1(\Omega)$ norm of the solutions is trivial for the defocusing NLS.  In the case of the CGLE, the energy functional (see \eqref{EtDef} and \eqref{EtDef2}) involves other terms such as $\frac{1}{p+1}(\alpha\kappa+\beta\lambda)\|u\|_{L^{p+1}(\Gamma_1)}^{p+1}$, $\alpha\lambda\int_0^t\|\Delta u\|_{L^2(\Omega)}^2ds$, and $\int_0^t\|u(s)\|_{L^{2p}(\Omega)}^{2p}ds$.  In order to achieve a similar type of control of the energy for the solutions of CGLE, we make an assumption on the sign of the frequency parameter.
\item The smoothing effect is also utilized here in obtaining global well-posedness results for a wider range of parameters $N$ and $p$.  The fact that $\lambda>0$ in the CGLE helps us to obtain better control estimates compared to the case of the nonlinear Schrödinger equation.  \cite{Dynamic_BC} obtains global well-posedness of strong solutions only in dimension $N=2$ and $p=3.$  We are able to improve this result in the context of the CGLE and obtain global well-posedness of strong solutions for  $p\ge 2$ if $N=1$; $p\in [2,5]$ if $N=2$; and $\displaystyle p\in \left[2,\frac{11}{3}\right]$ if $N=3$.
\item We show that solutions of the idbvp for NLS can be obtained as inviscid limits of solutions of the idbvp for CGLE as the parameter pair $(\lambda,\kappa)\rightarrow 0$.  This gives one another approach to study NLS with dynamic boundary conditions.  Inviscid limits and convergence to the NLS for the CGLE at different topological levels have been previously studied in the case of the whole space, periodic or homogeneous boundary conditions (see e.g., \cite{Bechouche}, \cite{Ogawa}, and \cite{Wu}).  We verify this property at the $H^1$-level under dynamic boundary conditions.  Moreover, we give a more clear regularity in time for $u_t$ when $u\in L^{\infty}(0,T;V\equiv H_{\Gamma_0}^1(\Omega)$) is the weak solution of the idbvp for the NLS.  More precisely, we prove that $u_t$ indeed belongs to $L^{2}(0,T;V')$ as opposed to $L^\infty(0,T;V')$ (the common space for the homogeneous boundary value problems).
\end{enumerate}

The rest of the paper is organized as follows:

\begin{enumerate}[(i)]
  \item Section 2 is devoted to the description of notations, assumptions, and statements of the main results.
  \item In Sections 3 and 4, we discuss the well-posedness of the corresponding linear homogeneous and nonhomogeneous systems.
  \item In Section 5, we study the Lipschitz perturbations of the linear equation with monotone boundary conditions.
  \item In Section 6, we obtain the local well-posedness of strong solutions for \eqref{GL_equation}.
  \item In Section 7, we study the global strong solutions.
  \item In Section 8, we discuss the existence and uniqueness of weak solutions.
  \item In Section 9, we prove that solutions of the idbvp for the NLS can be obtained as inviscid limits of the solutions of the idbvp for the CGLE.
  \item Finally, in Section 10, we study the long time behavior of solutions and obtain exponential decay rates by using a special multiplier, which is now a classical tool in the control theory of PDEs.
\end{enumerate}

\section{Notation and Main Results}\label{Not_Main}
We consider the space $L^2(\Omega)$ of complex valued functions on $\Omega$ endowed
with the inner product
$$(y,z)_{L^2(\Omega)}=\int_{ \Omega}  y(x)\overline{z}(x)\,dx$$
and the induced norm
$$||y||_{L^2(\Omega)}^2 = (y,y)_{L^2(\Omega)}.$$

We also consider the Sobolev space $H^1(\Omega)$ endowed with the scalar product
$$(y,z)_{H^1(\Omega)}= (y,z)_{L^2(\Omega)}+ (\nabla y,\nabla z)_{L^2(\Omega)}. $$

We will observe later that the natural underlying space for problem \eqref{GL_equation} is
\begin{equation}
V=\left\{u\in\,H^1(\Omega);\, u=0 \text{ on } \Gamma_0\right\}\end{equation} instead of the common $L^2(\Omega)$.  $V'$ will denote the dual space of $V.$

Since $\Gamma_0\neq \emptyset$, due to Poincaré's inequality, we can consider the space $ V $ endowed with the norm induced by the scalar product
\begin{equation}
\label{normV}\left(y,z\right)_V=\left(\nabla\,y,\nabla\,z\right)_{L^2(\Om)},\,\forall\,y, z\in\,V.
\end{equation} The norm $\|y\|_{V}\equiv\|\nabla y\|_{L^2(\Omega)}$  is equivalent to the usual norm of $ H^1(\Omega)\,. $\\

In what follows, we define the operator $A$ as the sum of the Schrödinger and heat operators, that is,
\begin{equation}\label{operator_A}
A \equiv (\lambda+i\,\alpha)\Delta
\end{equation}
with the domain given by
\begin{equation}\label{D(A)} D(A) \equiv \left\{y\in V, \Delta y\in V, \displaystyle\frac{\partial y}{\partial \nu} = -(\lambda+i\,\alpha)\Delta	y\text{ on }\Gamma_1   \right\}.
\end{equation}
In the above definition, $\Delta|_{\Gamma_1}$ should be interpreted as the restriction of the Laplacian from the interior to the boundary, which is well-defined according to the Sobolev trace theory since $\Delta y\in V$ for $y\in D(A)$.\\

We first study the linear idbvp corresponding to \eqref{GL_equation} with $g\equiv id$: \begin{equation}\begin{cases} u_t = (\lambda+i\,\alpha)\,\Delta u & \text{ in }\Omega\times (0,\infty),\\
u = 0 & \text{ on }\Gamma_0\times (0,\infty),\\
\displaystyle\frac{\partial u}{\partial \nu} = -u_t & \text{ on }\Gamma_1\times (0,\infty),\\
u(0)=u_0 &\text{ in }\Omega. \end{cases}\label{LSDyn}\end{equation}

The  operator $ A $ given in \eqref{operator_A} recasts the idbvp (\ref{LSDyn}) as the following Wentzell initial-boundary value problem (ibvp):
\begin{equation}\label{LSVen}\begin{cases} u_t = (\lambda+i\,\alpha)\Delta u & \text{ in }\Omega\times (0,\infty),\\
u = 0 & \text{ on }\Gamma_0\times (0,\infty),\\
\displaystyle\frac{\partial u}{\partial \nu} = -(\lambda+i\,\alpha)\Delta u & \text{ on }\Gamma_1\times (0,\infty),\\
u(0)=u_0 &\text{ in }\Omega. \end{cases}\end{equation}

In abstract operator theoretic form, we can rewrite \eqref{LSVen} alternatively as $$\dot{u}=Au, u(0)=u_0.$$
Using the above reformulation of problem \eqref{LSDyn}, we are able to prove the following well-posedness result.
\begin{theo}[Linear Homogeneous Problem I]\label{Semigroup_theo}  The operator $(A,D(A))$ generates a strongly continuous semigroup of contractions on $V$. \end{theo}

Going back to the idbvp \eqref{LSDyn}, one can restate the above theorem as follows.
\begin{theo}[Linear Homogeneous Problem II]\label{Linear}
	Let $u_0\in\,V$. Then there exists a unique solution $u\in\,C([0,\infty); V)$ to problem \eqref{LSDyn}.
\end{theo}
In order to deal with the nonlinear problem \eqref{GL_equation}, we first study the following non-homogeneous model:
\begin{equation}\begin{cases} u_t- (\lambda+i\,\alpha)\,\Delta u=f & \text{ in }\Omega\times (0,\infty),\\
u = 0 & \text{ on }\Gamma_0\times (0,\infty),\\
\displaystyle\frac{\partial u}{\partial \nu} = -u_t & \text{ on }\Gamma_1\times (0,\infty),\\
u(0)=u_0 &\text{ in }\Omega. \end{cases}\label{LSDynf}\end{equation}
The above idbvp, as in the linear homogeneous case, can be treated as a Wentzell ibvp.  Indeed, it is considered as a special case of the more general Wentzell problem below:
\begin{equation}\label{LSf}\begin{cases} u_t - (\lambda+i\,\alpha)\Delta u = f & \text{ in }\Omega\times (0,\infty),\\
u = 0 & \text{ on }\Gamma_0\times (0,\infty),\\
\ds\frac{\partial u}{\partial \nu} = -(\lambda+i\,\alpha)\Delta u+g & \text{ on }\Gamma_1\times (0,\infty)\\
u(0)=u_0  &\text{ in }\Omega\end{cases}\end{equation} with a given internal forcing term $ f $ and boundary input $ g. $

We prove the following result for problem \eqref{LSf}.
\begin{theo}[Linear Nonhomogeneous Problem I]\label{theo T} Let $f\in L^1(0,\infty ; V)$ and $g\in L^2(0,\infty ; L^2(\Gamma_1))$. Then for each $u_0\in V$ there exists a unique solution $u\in C([0,\infty) ; V)$ to (\ref{LSf}).
	Moreover, $ \Delta\,u\in\,L^2(0,\infty; L^2(\Omega)) $ and   the following ``hidden'' trace regularity holds true:
	$$\frac{\partial u}{\partial \nu}  \in L^2(0, \infty, L^2(\Gamma_1 ) ). $$
	\end{theo}
%
%
	
	The Wentzell ibvp in \eqref{LSf} can be formally identified with the  non-homogeneous idbvp \eqref{LSDynf} with the special choice $g \equiv -f|_{\Gamma_1}$.

	\begin{cor}[Linear Nonhomogeneous Problem II]\label{corollary_inho} Let $f\in L^2(0,\infty ; V)$. Then for each $u_0\in V$ there exists a unique solution $u\in C([0,\infty) ; V)$ to (\ref{LSDynf}).  Moreover, $u_t|_{\Gamma_1}
		\in L^2(0, \infty ; L^2(\Gamma_1) ), u_t
		\in L^2(0, \infty; L^2(\Om))$ and $ u
		\in L^2(0, \infty; H^2(\Omega) ). $ \end{cor}
\begin{rem}
 Due to the smoothing component of the operator $ A $ given in \eqref{operator_A}, the solution obtained here carries more regularity than the solution obtained in the case of the Schrödinger equation (compare Theorem \ref{theo T} and Corollary \ref{corollary_inho} with \cite[Theorem 1.4 and Corollary 1.5]{Dynamic_BC}).
\end{rem}

	We can  extend the linear theory to include Lipschitz perturbations (both on the interior and on the boundary) and nonlinear dynamic boundary feedback.
	\begin{equation}\label{NLSLip}\begin{cases} u_t = (\lambda+i\,\alpha)\Delta u + f(u) & \text{ in }\Omega\times (0,\infty), \\
	u = 0 & \text{ on }\Gamma_0\times (0,\infty), \\
	\ds\frac{\partial u}{\partial \nu} = -g(u_t) &\text{ on }\Gamma_0\times (0,\infty), \\
	u(0)=u_0 &\text{ in } \Omega.
	\end{cases} \end{equation}
	Here, $g(z)$ is assumed to satisfy the following conditions:
	\begin{assum}\label{Assumption_g}
		Assume that $g(z)$ is a continuous function on $\mathbb{C}$ such that both $g(z)$ and its inverse $g^{-1}(z)$ satisfy:
		\begin{enumerate}[(i)]
			\item $\text{Re} [(g(z)-g(v))(\bar{z}-\bar{v})] \geq m|z-v|^2$,
			\item $\text{Im} (g(z)\bar{z}) = 0$,
			\item $|g(z)| \leq M|z|$
		\end{enumerate}
for all $ v,z\in\,\mathbb{C} $ and 		for some constants $m,M\in\mathbb{R}_+$.
	\end{assum}

	Examples of functions satisfying Assumption \ref{Assumption_g} can be found in the literature for wave and  Schr\"odinger equations (see for example \cite{Lasiecka0}). In particular, assumptions (i) and (iii) form a complex analog to the assumption of monotonicity that appears in the study of wave equations.

When we consider the model in \eqref{NLSLip}, we will assume $f: V\to V$  to be Lipschitz continuous in the sense that there exists a constant $L$ such that for every pair $u,v\in V$,
	\begin{equation} \label{f}\|f(u)-f(v)\|_{V} \leq L\| u-v\|_{V}. \end{equation}

	Now, the well-posedness of \eqref{NLSLip} is achieved by
	considering a more general Wentzell ibvp. Namely, we replace $g(u_t)$ on the boundary with $g((\lambda+i\,\alpha)\Delta u + h(u)+\gamma\,u)$, where we assume $h: H^{1}(\Omega)\to L^2(\Gamma_1)$ is Lipschitz in the sense
	\begin{equation}\label{h} \|h(u)-h(v)\|_{H^{1/2}(\Gamma_1)} \leq K\| u-v\|_{V} \end{equation} for some $K>0.$
	
Since the trace operator $\gamma_0: H^1(\Omega)\to L^2(\Gamma)$ is continuous and linear, this formulation actually generalizes problem \eqref{NLSLip}, which is the special case where $h(u) = \gamma_0(f(y))$. In order to recast the problem in an abstract operator theoretic form, we define the operator $A_f$ given by
\begin{equation}\label{operator_Af} A_f u = (\lambda+i\,\alpha)\Delta u + f(u),\end{equation}
with the domain
\begin{equation}\label{D(A_f)} D(A_f) = \left\{y\in V, \Delta y\in V, \ds\frac{\partial y}{\partial \nu} = -g((\lambda+i\,\alpha)\Delta|_{\Gamma_1} y +h(y)) \text{ on }\Gamma_1   \right\}\end{equation}
where $g$ satisfies Assumption \ref{Assumption_g}, and $f$ and $h$ satisfy (\ref{f}) and (\ref{h}), respectively.
 One should notice that  if $f$ satisfies (\ref{f}), then $h \equiv \gamma_0 f $ satisfies (\ref{h}) by the Sobolev trace inequality.

 The following well-posedness fact holds true.
 \begin{theo}[Nonlinear Perturbations I]\label{D(Af)}  Taking into account Assumption \ref{Assumption_g} , (\ref{f}) and (\ref{h}),  the operator $(A_f,D(A_f))$ generates a strongly continuous semigroup on V. \end{theo}
Going back to the idbvp \eqref{NLSLip}, one obtains:
\begin{cor}[Nonlinear Perturbations II]\label{C:2}
Under the same assumptions in Theorem \ref{D(Af)},  for any initial data $u_0 \in V $ there exists a unique solution $u \in C([0, \infty) , V)$ of the problem (\ref{NLSLip}).
	
\end{cor}

Finally, we study the problem \eqref{GL_equation} with $ g\equiv id, $ that is,
\begin{equation}\label{GL_equation_1}
\begin{cases}
u_t-(\lambda+i\alpha)\triangle u+(\kappa+i\beta)|u|^{p-1}u-\gamma u=0 &\text { in } \Omega\times \mathbb{R_+},\\
\displaystyle\frac{\partial u}{\partial \nu} = -u_t &\text { on } \Gamma_1\times \mathbb{R_+},\\
u=0  &\text { on } \Gamma_0\times \mathbb{R_+},\\
u(0)=u_0  &\text { in } \Omega.
\end{cases}
\end{equation}

 	   In order to achieve our goal, we resort to the inhomogeneous linear theory   with a forcing term given by  $F(u) = -(\kappa+i\,\beta)|u|^{p-1}\,u+\ga\,u$. Using a contraction mapping argument and some a priori estimates, we prove the following  local well-posedness result:
 	   \begin{theo}[Local Strong Solutions]\label{T1}Let $N\le 3$ and $\beta>0$.  Then,
 	   	for every bounded subset $B\subset X_0$, there exists $T>0$ such that for all $(u_0, w_0)\in B$, there exists a unique solution $u$ of (\ref{GL_equation_1}) with time derivative $u_t = w$ such that the pair $(u,w)\in X_T$.
 	    \end{theo}
 	   Spaces $X_0$ and $X_T$ are defined in Section 6.  Given the association $w = u_t$, we can rephrase  $(u,w)\in X_T$ as
 	   \begin{equation} u \in C\left([0,T]; H^2(\Omega)\cap V\right)\cap C^1([0,T]; V) .\end{equation}

Regarding the global well-posedness, we have the following result:
 	   \begin{theo}[Global Strong Solutions]\label{GlobalThe}
Let $(u,u_t)\in X_T$ be a local strong solution as in Theorem \ref{T1} and $\beta>0$.  Then,  this solution can be extended globally under the conditions: $p\ge 2$ if $N=1$; $p\in [2,5]$ if $N=2$; and $\displaystyle p\in \left[2,\frac{11}{3}\right]$ if $N=3$.
 	    \end{theo}
Relaxing the smoothness assumption on $u_0$ a little bit, one can get solutions continuous in time at the $H^1$-level.  We first define the notion of weak solutions as follows.
\begin{defin}[Notion of Weak Solutions]\label{Defweak}We say that $ u $ is a weak solution of problem \eqref{GL_equation_1} if given $T>0$ and $$u_0\in\,Q\equiv\{\varphi\in V\text{ such that }\gamma_0\varphi\in L^{p+1}(\Gamma_1)\},$$  $u|_{t=0}=u_0$ and there exists a sequence of global strong solutions $u_\mu$ with initial data $u_{\mu,0}$ such that $ u_{\mu,0}\rightarrow u_0 $ in $Q$ and $u_\mu\rightarrow u$ in $ \,C([0,T],\,V)\,\cap\,L^2(0,T;\,H^2(\Om)) $, $u_{\mu}'\rightarrow u'$ in $\,L^2(0,T;\,L^2(\Om)) $, and $\gamma_0u_{\mu}'\rightarrow \gamma_0u'$ in $L^2(0,T;\,L^2(\Gamma_1))$.
	 			 	\end{defin}

Our next result is  the following theorem on the existence and uniqueness of weak solutions:
	 			 	\begin{theo}[Weak Solutions]\label{main_density_theo}Let $N\le 3$, $\beta>0$, $ u_0\in\,V  $ so that $\gamma_0u_0\in L^{p+1}(\Gamma_1)$ and $ (p,N) $ satisfy the conditions given in Theorem 2.6.  Then, problem \eqref{GL_equation_1} possesses  a unique (weak) solution in the sense of Definition \eqref{Defweak}.
	 			 	\end{theo}
	 			 	
	 			 	\begin{rem}
	 			 		Theorem \eqref{main_density_theo} gives a more regular weak solution than that obtained in the corresponding Schrödinger problem in \cite{Dynamic_BC}, where the weak solution only satisfies $ u\in L^\infty(0,T;\,V) $ with $ u^\prime\in\,L^2(0,T;\,V^\prime). $ Moreover, the uniqueness in NLS is proved only in the case that the nonlinearity is globally Lipschitz in $ V $, whereas for the CGLE, uniqueness is proved in a more general setting, thanks to  the smoothing effect of the parabolic component of the Ginzburg-Landau operator.
	 			 	\end{rem}
At this point, a natural question is to ask whether the solutions of the CGLE with dynamic boundary conditions get close to the solutions of NLS with same type of boundary conditions as the parameter pair $\epsilon\equiv (\lambda,\kappa)\rightarrow 0$.  We show that this is indeed true by the following theorems.

       \begin{theo}[Inviscid Limits I]\label{Inviscid0}
       Suppose that $u_\epsilon$, where $\epsilon=(\lambda,\kappa)$, is a global (weak) solution to the idbvp for the CGLE with the initial condition $u_0\in Q$ as in Theorem \ref{main_density_theo}.  Then, there exists
       $u\in L^{\infty}(0,T;V)$ with $u_t\in L^{2}(0,T;V')$ such that a subsequence of $u_\epsilon$ (still denoted same) satisfies
            \begin{eqnarray}
            u_\epsilon &\stackrel{\hspace{-1mm}}{\rightharpoonup} &u \mbox{ weakly star in}
	 	   \,L^\infty(0,T;\,V)\,,\\
             \partial_t u_\epsilon &\stackrel{\hspace{-1mm}}{\rightharpoonup} &u_t \mbox{ weakly in}
	 	   \,L^2(0,T;\,V')\,
	 	   \end{eqnarray} as $\epsilon\rightarrow 0$, and most importantly $u$ solves the idbvp for the NLS in the weak sense.

       \end{theo}

 	   \begin{theo}[Inviscid Limits II]\label{Inviscid}
Let $N=2$ and $p=3$.  Suppose that $u_\epsilon$ is a global strong solution of the idbvp for the CGLE with the initial condition $u_\epsilon^0$ and $u$ is a global strong solution of the idbvp for the NLS with initial condition $u_0$ such that $u_\epsilon^0\rightarrow u_0$ in $V$ as $\epsilon=(\lambda,\kappa)\rightarrow 0$.  Then, $$\|u_\epsilon-u\|_V=O(\lambda)+O(\kappa)$$ as $\epsilon=(\lambda,\kappa)\rightarrow 0.$
 	    \end{theo}

Finally, we prove that the solutions of the idbvp for the CGLE decay to zero exponentially fast if $\gamma\le 0$.  This is easy to prove with $\gamma<0$, and we have the following theorem.

 	   \begin{theo}[Stabilization I]\label{decaythm}Let $$u_0\in Q\equiv\{\varphi\in V\cap L^{p+1}(\Omega) \text{ such that }\gamma_0\varphi\in L^{p+1}(\Gamma_1)\}$$ and
$u$ be the corresponding global weak solution of the idbvp for the CGLE (eq. \ref{GL_equation_1}) with $\gamma<0$ as in Theorem \ref{main_density_theo}.  Then, $$F(t)\le E_0e^{-|\gamma| t} \text{ for } t\ge 0,$$ where \begin{equation}\label{defFt}F(t)\equiv \frac{\alpha}{2}\|\nabla u(t)\|_{L^2(\Omega)}^2+\frac{\beta}{p+1}\|u(t)\|_{L^{p+1}(\Omega)}^{p+1}\end{equation} and
\begin{multline}E_0\equiv \frac{\alpha}{2} \|\nabla u_0\|_{L^2(\Omega)}^2 +\frac{\beta}{p+1}\|u_0\|_{L^{p+1}(\Omega)}^{p+1} -\frac{\alpha\gamma}{2}\|u_0\|_{L^2(\Gamma_1)}^2\\+\frac{1}{p+1}(\alpha\kappa+\beta\lambda)\|u_0\|_{L^{p+1}(\Gamma_1)}^{p+1}.\end{multline}
 	    \end{theo}

\begin{rem} The problem $\ga=0$ is more challenging and requires control theoretic tools. In the case $\gamma=0$, there is usually no decay for the complex Ginzburg-Landau equation even at $L^2$-level.  However, in our model the dynamic boundary input plays the role of a stabilizing control/feedback and one actually gains an exponential decay of solutions.  In fact, we have the following theorem.\end{rem}

 	   \begin{theo}[Stabilization II]\label{decaythm2}
Let $$u_0\in Q\equiv\{\varphi\in V\cap L^{p+1}(\Omega) \text{ such that }\gamma_0\varphi\in L^{p+1}(\Gamma_1)\}$$ and
$u$ be the corresponding global weak solution of the idbvp for the CGLE (eq. \ref{GL_equation_1}) with $\gamma=0$ as in Theorem \ref{main_density_theo}.  Moreover, suppose that $\Omega$ satisfies the following geometric condition:  $\exists x_0\in\mathbb{R}^N$ such that $(x-x_0)\cdot \nu\le 0$ on $\Gamma_0$ and $(x-x_0)\cdot \nu> 0$ on $\Gamma_1$.
Then, there exists some $C>0$ such that $$F(t)\le F(0)e^{1-\frac{t}{C}} \text{ for } t\ge 0$$  where $F(t)$ is given in \eqref{defFt}.
 	    \end{theo}
\section{Linear Homogeneous Problem}
\setcounter{equation}{0}
In this section, our aim is to prove Theorems \ref{Semigroup_theo} and \ref{Linear} by using similar arguments given in \cite{Dynamic_BC}.  In order to achieve this, we convert the idbvp \eqref{LSDyn} into the Wentzell ibvp in \eqref{LSVen}. To this end, we shall prove that the operator $ A $ is maximal dissipative.\\

\noindent{\bf Dissipativity:}

 The key difficulty is that this operator is not dissipative on the most natural space $L^2(\Omega). $ Indeed, from the definition of the operator $ A $ given in \eqref{operator_A}, we have:
 \begin{equation} \begin{aligned}
 (Au,u)_{L^2(\Omega)} &= (\lambda+i\,\alpha)(\Delta u, u)_{L^2(\Omega)} \\&= -(\lambda+i\,\alpha)(\nabla u, \nabla u)_{L^2(\Omega)} + (\lambda+i\,\alpha)\left(\displaystyle
 {\frac{\partial u}{\partial \nu} }, u\right)_{L^2(\Gamma_1)}
 \end{aligned}\end{equation} for $u \in D(A).$
 Since $ \displaystyle\frac{\partial u}{\partial \nu}= -(\lambda+i\,\alpha)\,\Delta\,u $ on $ \Gamma_1$ (see (\ref{LSVen})), we	get
 \begin{equation} \text{Re} (Au,u)_{L^2(\Omega)} = -\lambda\|\nabla u\|_{L^2(\Omega)}^2-\text{Re}[(\lambda+i\,\alpha)^2\,(\Delta u, u)_{L^2(\Gamma_1)}].\end{equation} Unfortunately, it is not clear from the above inequality that $\text{Re} (Au,u)_{L^2(\Omega)}\le 0$.	Therefore, we cannot say that $A$ is dissipative in $L^2(\Omega).$
 This brings forward the idea used in the past \cite{Goldstein} to treat the problem within the $H^1$ topology instead.

 So,  let  us do the same computation with respect to the scalar product of the space $V$ now.  	Then, for any $u \in D(A)$,	we have
 \begin{equation} \begin{aligned}
(Au, u)_V&= (\nabla Au, \nabla u)_{L^2(\Omega)}\\ &= (\lambda+i\,\alpha)(\nabla\Delta u, \nabla u)_{L^2(\Omega)}\\
 &= -(\lambda+i\,\alpha)(\Delta u, \Delta u)_{L^2(\Omega)} + (\lambda+i\,\alpha)
 \left(\Delta u, \displaystyle\frac{\partial u}{\partial \nu}\right)_{L^2(\Gamma_1)},
 \end{aligned}\end{equation}
 whereby substituting $\displaystyle\frac{\partial u}{\partial \nu} = -(\lambda+i\,\alpha)\Delta u$
 on the boundary	$ \Gamma_1 $, we obtain
 \begin{equation}\label{identAu} ( Au,  u)_{V} = -(\lambda+i\,\alpha)\|\Delta u\|_{L^2(\Omega)}^2
 - \left\|\displaystyle\frac{\partial u}{\partial \nu}\right\|_{L^2(\Gamma_1)}^2.\end{equation}
 Hence, taking the real parts in \eqref{identAu} and   making use of the Cauchy-Schwarz inequality, we have
 \begin{align}\label{identAu1} \text{Re}( Au,  u)_{V}&=-\lambda\,\|\Delta u\|_{L^2(\Omega)}^2-
 \left\|\displaystyle\frac{\partial u}{\partial \nu}\right\|_{L^2(\Gamma_1)}^2\le 0.\end{align}

  This	proves	the	dissipativity	of	the	operator	$ A. $\\


 \noindent{\bf Maximallity:}

We	will define a bilinear form $a(\cdot,\cdot)$ in such a way that the identity
 \begin{equation}\label{avz def}a(u,z) = (-Au +\theta\, u, z)_V \end{equation} will hold true whenever $u\in D(A)$ and $z\in V$, where $\theta$ will be chosen later.	An appropriate definition for such a bilinear form is
 \begin{equation}\label{bilinearform}a(u,z)\equiv (\lambda+i\alpha)(\Delta u,\Delta z)_{L^2(\Omega)}+\left(\frac{\partial u}{\partial \nu},\frac{\partial z}{\partial \nu}\right)_{L^2(\Gamma_1)}+\theta(\nabla u,\nabla z)_{L^2(\Omega)},\end{equation} taking into consideration the definition of $A$ and integrating by parts.

 Now, we	introduce	the	space	\begin{equation}\label{space Z} Z = \left\{z\in V, \Delta z\in L^2(\Omega), \displaystyle
 \frac{\partial z}{\partial \nu}\in L^2(\Gamma_1)\right\}, \end{equation}
 which is equipped  with the norm
 \begin{equation}\label{norm Z} \|z\|_Z^2 = \|z\|^2_V + \|\Delta z\|^2_{L^2(\Omega)} + \left\|\displaystyle\frac{\partial z}{\partial \nu}\right\|^2_{L^2(\Gamma_1)} .\end{equation}

$ Z $	is	indeed a	Banach	space (see \cite[Lemma 2.1]{Dynamic_BC} for a proof). We will consider the bilinear form \eqref{bilinearform} on this Banach space.

Our next step is to invoke the Browder-Minty theorem \cite[Theorem 5.16]{Brezisbook} to show that for any fixed $f\in V$,
 there exists a unique (weak) solution $u\in V$ satisfying the following variational form
 \begin{equation*} a(u,z) = (-f,z)_V, \end{equation*}
 for all $z\in V$.

 This is done by showing that $a(u,z)$ is continuous and coercive on $Z$. Observe that
 \begin{eqnarray} \label{bridge avz} a(u,z) &=& -(\lambda+i\,\alpha)(\Delta u, z)_V +\theta\,(u,z)_V \\
 &=& (\lambda+i\,\alpha)(\Delta u, \Delta z)_{L^2(\Omega)} -(\lambda+i\,\alpha)\left(\Delta u, \ds\frac{\partial z}{\partial \nu}\right)_{L^2(\Gamma_1)} +\theta (u,z)_V.\nonumber\end{eqnarray}

 Now, since $\ds\frac{\partial u}{\partial \nu}=-(\lambda+i\,\alpha)\,\Delta\,u$ on $\Gamma_1,$
 \eqref{bridge avz} can be rewritten as
 { \begin{equation}\begin{aligned}\label{bridge avz1}   a(u,z)
 	= (\lambda+i\,\alpha)(\Delta u, \Delta z)_{L^2(\Omega)} +\left(\ds\frac{\partial u}{\partial \nu}, \ds\frac{\partial z}{\partial \nu}\right)_{L^2(\Gamma_1)}+\theta (u,z)_V ,\end{aligned}\end{equation}}

 so we obtain
 \begin{align} |a(u,z)| &\leq C(\lambda, \alpha)\left|(\Delta u, \Delta z)_{L^2(\Omega)}\right|
 +  \left| \left(\ds\frac{\partial u}{\partial \nu}, \ds\frac{\partial z}{\partial \nu}\right)_{L^2(\Gamma_1)}\right|+ |\theta| \left|(u,z)_V\right|. \end{align}

 Applying the Cauchy-Schwarz inequality to each of the respective inner products above, one gets the estimate
 \begin{equation}  |a(u,z)| \leq C(\alpha, \lambda, \theta
 )\|u\|_Z \|z\|_Z,\end{equation}
 which proves the continuity of $a(u,z)$. \\

 In order to prove the coercivity of $a(u,z)$, we write $z=u$ in \eqref{bridge avz1}. It follows that
 {\small\begin{equation}\label{coercive} \begin{aligned}|a(u,u)| &= \left|\theta\|u\|_V^2 +(\lambda+i\,\alpha)\|\Delta u\|_{L^2(\Omega)}^2
 	+\left\|\ds\frac{\partial u}{\partial \nu}\right\|_{L^2(\Gamma_1)}^2\right| .\end{aligned}\end{equation}}
 On the other hand, we have
 {\small\begin{equation*}
 	\begin{aligned}
 	\Re                 a(u,u)=&\Re(\theta)\,\|u\|_V^2+\lambda\,\|\Delta\,u\|_{L^2(\Om)}^2+\left\|\ds\frac{\partial u}{\partial \nu}\right\|_{L^2(\Ga_1)}^2,\\
 	\Im a(u,u)=&\Im(\theta)\,\|u\|_V^2+\alpha\,\|\Delta\,u\|_{L^2(\Om)}^2.
 	\end{aligned}
 	\end{equation*}}
Now, making use of the algebraic inequality below
 {$$\sqrt{2}\,|z|\geq\,|\Re\,z|+|\Im\,z|,\,\forall\,z\in\,\mathbb{C},
 	$$} it follows that
 {\small\begin{equation}\label{auu}\begin{aligned}
 	|a(u,u)|&\geq\frac{\sqrt{2}}{2}\left(\left|\Re(\theta)\,\|u\|_V^2+\lambda\,\|\Delta\,u\|_{L^2(\Om)}^2+\left\|\ds\frac{\partial u}{\partial \nu}\right\|_{L^2(\Ga_1)}^2
 	\right|\right.\\
 	&+\left.\left|\Im(\theta)\,\|u\|_V^2+\alpha\,\|\Delta\,u\|_{L^2(\Om)}^2\right|\right) \,.\end{aligned}
 	\end{equation}}
 Now, recalling the triangle inequality
 $ |a+b|\leq |a|+|b| $ with  {\small\begin{eqnarray*}a&=& \Re(\theta)\,\|u\|_V^2+\lambda\,\|\Delta\,u\|_{L^2(\Om)}^2+\left\|\ds\frac{\partial u}{\partial \nu}\right\|_{L^2(\Ga_1)}^2,
 		\\
 		b&=& \Im(\theta)\,\|u\|_V^2+\alpha\,\|\Delta\,u\|_{L^2(\Om)}^2\end{eqnarray*} and considering $\Im(\theta)\geq 0,$ from \eqref{auu}, we infer
 	{\small\begin{equation}
 		\begin{aligned}
 		\label{coercive1}&|a(u,u)|\geq\,\frac{\sqrt{2}}{2}\left|(\Re(\theta)+\underbrace{\Im(\theta)}_{\geq 0})\|u\|_V^2+\left\|\ds\frac{\partial u}{\partial \nu}\right\|_{L^2(\Ga_1)}^2+(\lambda+
 		\alpha)\|\Delta\,u\|_{L^2(\Om)}^2\right|\\
 		&\geq
 		\frac{\sqrt{2}}{2}\left(\Re(\theta)\|u\|_V^2+\left\|\ds\frac{\partial u}{\partial \nu}\right\|_{L^2(\Ga_1)}^2+(\lambda+
 		\alpha)\|\Delta\,u\|_{L^2(\Om)}^2\right).
 		\end{aligned}
 		\end{equation}}
 	
 	By taking, $\Re(\theta)>0$ we prove the coercivity of $A $.

 	We conclude from the complex version of the Browder-Minty theorem that for all $f\in Z'$, where $Z'$
 	denotes the dual space of $Z$, there exists a solution $v\in Z$ to $a(v,z) = (-f,z)_V$.
 	Moreover, we observe that $D(A)\subset Z\subset V\subset Z'$; hence for all $f\in V$ there is a solution
 	$v\in Z\subset V$. Furthermore, by testing the variational form with $z\in Z$ satisfying $\displaystyle \frac{\partial z}{\partial \nu}=0$, it follows that
 	$$(\lambda+i\alpha)(\Delta v,\Delta z)_{L^2(\Omega)}-\theta(v,\Delta z)_{L^2(\Omega)}=(f,\Delta z)_{L^2(\Omega)},$$ whence
 	\begin{equation*} (\lambda+i\,\alpha)\Delta v - \theta v = f.\end{equation*}  Since $v, f\in V$, now it follows that $\Delta v\in V$. Using the variational form once more, one recovers the boundary condition so that $v\in D(A)$. This	tells us	that the operator  $A$ is the infinitesimal generator of a $C_0$ semigroup of contractions on $V.$

Moreover, $\Delta v\in V$ implies that $\Delta v|_{\Gamma_1} \in H^{1/2}(\Gamma_1)$
 	and thus $\ds\frac{\partial v}{\partial \nu}\in H^{1/2}(\Gamma_1)$ as well.
 	Trace theory tells us that $v\in H^2(\Omega)$; thus we know that the regularity of $D(A)$ is at least at the level of $H^2(\Omega)$. Since $ D(A) $ is dense in $V$,
 	from the Lummer -- Philips	theorem or \cite[Corollary IV.3.2,]{Show}, Theorem \ref{Semigroup_theo} follows. Translating the Wentzell ibvp \eqref{LSVen} back into the idbvp \eqref{LSDyn}, the proof Theorem \ref{Linear} is completed.

 	\section{Linear Nonhomogeneous Problem}
 	\setcounter{equation}{0}
 	In this section, our aim is to prove Theorem \ref{theo T} and Corollary \ref{corollary_inho}.  We now consider the problem \eqref{LSDynf}. First of all, according to \cite[Theorem IV.4.1A]{Show}, we can deduce the following result: \begin{prop}\label{LSDynf_prop}
 		If $ f\in\,L^1(0,\infty; V) $ and $ u_0\in\,\overline{D(A)}=V $, there exists a unique generalized solution $ u\in\,C([0,\infty); V) $ to the problem \eqref{LSDynf}.
 	\end{prop}

 	Now, we want to consider the general Wentzell ibvp in \eqref{LSf}.  Let us first consisder the case $f = 0$ and then employ the superposition principle to obtain the well-posedness of \eqref{LSf}. First, we define the Neumann map $\mathcal{	N}$ below where for a given $g\in H^s(\Gamma_1)$, $\mathcal{N}g$ solves
 	\begin{equation}
 	\begin{cases}
 	\Delta \mathcal{N}g = 0 \\
 	\dis\frac{\partial \mathcal{N}g}{\partial\nu}  = g  &\text{ on } ~ \Gamma_1 \\
 	\mathcal{N}g =0  &\text{ on } ~ \Gamma_0.
 	\end{cases}\label{Neumann}
 	\end{equation}
 	Using the elliptic theory \cite{Lions-Magenes}, one can deduce
 	\begin{equation}\label{Neumann_cont}
 	\mathcal{N}: H^s(\Gamma_1)\rightarrow H^{s+3/2}(\Omega) \quad \text{ is continuous for all $s\in\mathbb{R}$.}
 	\end{equation}
 	
 	Define now $\tilde{u} = u-\mathcal{N}g$. Using the boundary condition on $ \Ga_1 $ given in \eqref{LSf}, $\Delta \mathcal{N}g = 0$ and $ \dis\frac{\partial \mathcal{N}g}{\partial\nu}=g $ on $ \Ga_1, $ it follows that
 	\begin{equation}\label{neubd}\begin{aligned}
 	\frac{\partial \tilde{u}}{\partial\nu} =&\frac{\partial{u}}{\partial\nu}-\frac{\partial \mathcal{N}\,g}{\partial\nu} \\
 	=&\cancel{g}-(\la+i\,\al)\,\Delta\,u-\cancel{g}\\
 	=&-(\la+i\,\al)\,\Delta\,u+(\la+i\,\al)\,\underbrace{\Delta\mathcal{N}g}_{=0}\\
 	=&-(\la+i\,\al)\,\Delta[\,u-{\mathcal{N}g}]\\
 	=&-(\la+i\,\al)\,\Delta \tilde{u} \quad \text{ on }\,\Ga_1.
 	\end{aligned} \end{equation}
 	
 	Moreover, since we are considering the case  $ f=0, $ we see that
 	\begin{equation*} \begin{aligned}
 	\tilde{u}_t =& u_t - \mathcal{N}g_t\\
 	=& (\la+i\,\al)\,\Delta u - \mathcal{N}g_t.
 	\end{aligned}\end{equation*}
 	Again making use of the fact that  $\Delta \mathcal{N}g = 0$, we have
 	\begin{equation}\label{neuint}\begin{aligned}
 	\tilde{u}_t =& (\lambda+i\,\alpha)\,\Delta (u - \mathcal{N}g) -\mathcal{N}g_t\\ =& (\lambda+i\,\alpha)\,\Delta\tilde{u} -\mathcal{N}g_t.
 	\end{aligned}\end{equation}
 	Combining (\ref{neubd}) and (\ref{neuint}), the problem with respect to the function  $\tilde{u}$ becomes:
 	\begin{equation}\label{tildeprob}\begin{cases} \tilde{u}_t  = (\lambda+i\,\alpha)\,\Delta\tilde{u} -\mathcal{N}g_t& \text{ in }\Omega\times (0,\infty) \\
 	\tilde{u} = 0 & \text{ on }\Gamma_0\times (0,\infty) \\
 	\displaystyle{\frac{\partial \tilde{u}}{\partial \nu}} = -(\lambda+i\,\alpha)\,\Delta \tilde{u} & \text{ on }\Gamma_1\times (0,\infty) \\
 	\tilde{u}(0) = u_0 - \mathcal{N}g(0) .\end{cases}\end{equation}

 	The above problem gives us the following result:
 	
 	\begin{lem}\label{tildecont} If $g\in W^{1,1}(0,\infty ; H^{-1/2}(\Gamma_1))$, then there exists a unique solution  ${u}\in {C([0,\infty) ;V)}$ to the problem (\ref{LSf}). \end{lem}
 	\begin{proof} Indeed, if   $g\in  W^{1,1}(0,\infty ; H^{-1/2}(\Gamma_1))$, then $g_t\in L^1(0,\infty ; H^{-1/2}(\Gamma_1))$ and  $g(0)\in H^{-1/2}(\Gamma_1)$. Using these and taking into account that the Neumann map given in  (\ref{Neumann_cont}) is continuous from $ H^{-1/2}(\Gamma_1)$ into $H^{1}(\Omega)$, it follows that
 		\begin{equation}\begin{cases} -\mathcal{N}g_t \in L^1(0,\infty ; V) \\
 		\mathcal{N}g(0) \in V.\end{cases}\end{equation}
 		Using Proposition \ref{LSDynf_prop} we obtain the well-posedness of \eqref{tildeprob} and that  $\tilde{u} \in C([0, \infty); V ) $. Writing $u=\tilde{u}+\mathcal{N}g$ and using the fact $\tilde{u},\, \mathcal{N}g\in C([0,\infty) ; V)$, we conclude that there exists a unique solution  $u\in C([0,\infty) ; V)$ to the problem (\ref{LSf}).
 	\end{proof}

Below we will prove Theorem \ref{theo T}. This  result requires only $g\in L^2(0,T; L^2(\Gamma_1))$.  Indeed, Lemma \ref{tildecont} tells us that that there exists a unique solution  $u\in {C([0,T] ; V)}$  to the problem (\ref{LSf}) whenever  {$g\in  W^{1,1}(0,T ; H^{-1/2}(\Gamma_1))$}. So it is sufficient to prove that $\ds\sup_{t\in[0,T]} \|u(t)\|_V < \infty$ for  $g\in L^2(0,T; L^2(\Gamma_1))$ and then use a density argument.
 		
 		Multiplying (\ref{LSf}) by $\bar{u}$ in $V$, taking the inner product, and integrating in temporal variable  $ t,$ it follows that
 		{\small	\begin{equation}\label{energy11} \begin{aligned}
\ds\int_0^t (u_t(s),u(s))_V \, ds - \ds\int_0^t ((\lambda+i\,\alpha)\Delta u(s),u(s))_V \, ds = \ds\int_0^t (f(s),u)_V \, ds .
 			\end{aligned}\end{equation}} Without loss of generality,  we take $f = 0$. The  general case  $f\neq 0$ can be solved a posteriori via superposition. We observe that the first term satisfies
 		\begin{equation}\label{energy12} \ds\text{Re}\int_0^t (u_t(s),u(s))_V \, ds = \frac{1}{2}\ds\int^t_0 \frac{d}{dt} \|u(s)\|^2_V \, ds = \frac{1}{2}\|u(t)\|^2_V - \frac{1}{2}\|u(0)\|^2_V .\end{equation}
 		Using integration by parts on the second term at the left hand side of  (\ref{energy11}) we get
 		\begin{equation} \begin{aligned}
 		&	\ds\int^t_0 ((\lambda+i\,\alpha)\Delta u(s), u(s))_V\, ds \\&= \ds\int^t_0 \left[-(\lambda+i\,\alpha)\|\Delta u(s)\|^2_{L^2(\Omega)} +(\lambda+i\,\alpha)\left(\Delta u(s), \frac{\partial u}{\partial\nu}(s) \right)_{L^2(\Gamma_1)}\right]\, ds.
 		\end{aligned} \end{equation}
 		Now, substituting the boundary condition on $ \Ga_1 $, namely,	 $\displaystyle (\lambda+i\,\alpha)\,\Delta\,u=g-\frac{\partial u}{\partial \nu}$ we obtain
{\small 		\begin{equation} \label{delta}\begin{aligned}
 		&\ds\int^t_0 \left((\lambda+i\,\alpha)\Delta u(s), \frac{\partial u}{\partial\nu}(s)\right)_{L^2(\Ga_1)} \, ds\\& =\ds\int^t_0 \left[-\left\|\frac{\partial u}{\partial \nu} (s)\right\|^2_{L^2(\Gamma_1)} + \left(g, \frac{\partial u}{\partial\nu}(s)\right)_{L^2(\Gamma_1)}\right]\, ds .
 		\end{aligned}\end{equation}}
Substituting \eqref{energy12} -- \eqref{delta} in \eqref{energy11} and taking the real parts, we get:
 		{\small	\begin{eqnarray}
 			\label{4.17rewritten} \begin{aligned}
 			0&=			\frac{1}{2}\|u(t)\|^2_V - \frac{1}{2}\|u(0)\|^2_V + \ds\lambda\int^t_0 \|\Delta u(s)\|^2_{L^2(\Omega)}\,ds+\ds \int^t_0\left\|\frac{\partial u}{\partial \nu} (s)\right\|^2_{L^2(\Gamma_1)}\,ds\\ &-\Re \int_0^t\left(g, \frac{\partial u}{\partial\nu}(s)\right)_{L^2(\Gamma_1)}\, ds.
 			\end{aligned}\end{eqnarray}}
 		
We infer that

 		\begin{equation}\label{4.17ineq0} \begin{aligned}
 		&\frac{1}{2}\|u(t)\|^2_V + \ds\lambda\int^t_0 \|\Delta u(s)\|^2_{L^2(\Omega)}\,ds+ \left(1-\frac{1}{4\,\eta}\right)\ds\int_0^t \left\|\frac{\partial u}{\partial \nu} (s)\right\|^2_{L^2(\Gamma_1)}\, ds\\
 		&\leq  \frac{1}{2}\|u(0)\|^2_V + \eta\, \ds\int_0^t \| g\|^2_{L^2(\Gamma_1)}\, ds
 		\end{aligned}\end{equation} where we choose $\eta>\frac{1}{4}$.

 		More generally, for $f\in L^1(0,T; V)$, we observe
 		\begin{equation}
 		\label{Gronwall2}\begin{aligned}
 		&\frac{1}{4}\|u(t)\|^2_V + \ds\lambda\int^t_0 \|\Delta u(s)\|^2_{L^2(\Omega)}\,ds+ \left(1-\frac{1}{4\,\eta}\right)\ds\int_0^t \left\|\frac{\partial u}{\partial \nu} (s)\right\|^2_{L^2(\Gamma_1)}\, ds\\ &\leq \frac{1}{2}\|u(0)\|^2_V + \ds\eta\int_0^t \| g\|^2_{L^2(\Gamma_1)}\, ds+\ds\int_0^t \|f\|_V^2\, ds\\
 		&\leq C(u_0, f, g).
 		\end{aligned}
 		\end{equation}

 	 	By making the identification $g= f|_{\Gamma_1}$, we can identify (\ref{LSf}) with the idbvp \eqref{LSDynf}. 	Note that if $f\in L^2(0,T ; V)$, by trace theory $g\in L^2(0,T ; H^{1/2}(\Gamma_1))$. Moreover, since $ \Delta\,u\in\,L^2(0,T; L^2(\Om)) $ and $ f\in L^2(0,T ; V), $ we have that $ u_t\in L^2(0,T ; L^2(\Om) ) $, and this concludes the proof of Corollary \ref{corollary_inho}.
 	
 	\begin{rem}

 \begin{enumerate}[(i)]
               \item Given $u\in H^1(\Omega)$, by  trace theory applied only {\it formally}, one obtains that $\displaystyle\frac{\partial u}{\partial \nu}\in H^{-1/2}(\Gamma)$. Therefore, the additional regularity $\dis \frac{\partial u}{\partial \nu}\in L^2(\Gamma_1)$, which shows up in Theorem \ref{theo T}, is a ``hidden''
 		regularity due to the underlying Wentzell structure.
               \item Again, just by a formal argument, one would expect $ \Delta\,u\in\,H^{-1}(\Omega)$.  The fact that $ \Delta\,u\in\,L^2(\Om), $ is a result of the smoothing effect due to the intrinsic properties of the parabolic component of the Ginzburg-Landau operator.
               \item From Corollary \ref{corollary_inho}, we also conclude that the map
 	\begin{equation} (f,u_0) \longmapsto u \end{equation}
 	is bounded from  $L^2(0,\infty, V)\times V$ into $C([0,\infty); V)$.
             \end{enumerate}
  \end{rem}

 	Next, we desire to obtain a Duhamel's formula  for a problem derived from \eqref{LSf} by following the ideas in \cite{Dynamic_BC}.
 	Our goal is to obtain an estimate for $ u_t .$ To this end,  we will first deduce a semigroup representation for $ u $.\\
 	
 	\begin{lem}\label{l:sem}
 		Let $ u_0, f $ and $ g $ satisfy
 		\begin{enumerate}\begin{multicols}{2}
 				
 			\item[(i)]$ u_0\in\,V $
 			\item[(ii)] $ f\in\,L^1(0,T; V) $
 			\item[(iii)] $ g\in\,L^2(0,T; L^2(\Gamma_1)) $.
 			\end{multicols}
 		\end{enumerate}
 		Then, the solution $ u\in\,C([0,T];V) $ of \eqref{LSf} can be represented by
 		\begin{equation}\label{Semigroup}
 		u(t) = e^{At} u_0 - A \int_0^t e^{A(t-s) }{\mathcal{N}} g (s) ds + \int_0^t e^{A(t-s) } {f}(s) ds.
 		\end{equation}
 	\end{lem}
 	\begin{proof}
 		From Duhamel's formula, the solution of \eqref{tildeprob} is given by
 		$$\tilde{u}(t) = e^{At} \tilde{u}(0)  - \int_0^t e^{A(t-s) }{\mathcal{N}} g _t(s) ds + \int_0^t e^{A(t-s) }{f}(s) ds.
 		$$The above formula is understood with the values in the dual space $[D(A)]' $ .
 		Integrating by parts yields \begin{equation}\begin{aligned} \tilde{u}(t) &=   e^{At} \tilde{u}(0)- \left.e^{A(t-s)}{\mathcal{N}}g(s)\right|_{0}^t-A \int_0^t e^{A(t-s) }\widehat{\mathcal{N}} g (s) ds + \int_0^t e^{A(t-s) } {f}(s) ds\\&=e^{At} \tilde{u}(0)  - {\mathcal{N}} g(t) + e^{At} {\mathcal{N}} g(0)  -A \int_0^t e^{A(t-s) }{\mathcal{N}} g (s) ds + \int_0^t e^{A(t-s) } {f}(s) ds.\end{aligned}\end{equation}
 		Since  $ \tilde{u}= u - {\mathcal{N}}g $, we obtain the desired representation given in  (\ref{Semigroup}).\\

 		Then, from assumptions (i) and (ii), the first and the third term  of \eqref{Semigroup} belong to $ C([0,T]; V). $ To prove that $u\in\,C([0,T]; V)  $, it remains to be shown that $ A \int_0^t e^{A(t-s) }{\mathcal{N}} g (s) ds\in\, C([0,T];V). $ In fact, combining assumption (iii) with \eqref{Neumann_cont} for $ s=0 $, it results that $ \mathcal{N}g\in\,L^2(0,T,V). $  Hence, due to \cite[Theorem 2.4(b), pg.5]{Pazy}, we obtain that the term  $ \int_0^t e^{A(t-s) }{\mathcal{N}}g(s)\,ds\in\,D(A), $ which allows us to say that  $ A\,\int_0^t e^{A(t-s) }{\mathcal{N}}g(s)\,ds $ makes sense.
 	\end{proof}
 	We deduce from the proof of Theorem \ref{theo T} that the following map is continuous:
 	\begin{equation}\label{Lmap}
 	\begin{aligned}
 	\mathcal{L}:L^2(0,T; L^2(\Gamma_1) )&\rightarrow C ([0,T]; V )\\
 	g\hspace{1.4cm}&\mapsto A \int_0^t e^{A(t-s) }{\mathcal{N}} g (s) ds.
 	\end{aligned}
 	\end{equation}

 	Below, we will be looking at more regular solutions corresponding  to the inhomogeneous problem (\ref{LSf}). We start our discussion with the following result:
 	\begin{theo}
 		In addition to the assumptions in Lemma \ref{l:sem} assume also that:
 		\begin{enumerate}[(i)]\begin{multicols}{2}
 				
 			\item
 			$f_t \in L^1(0,T; V) $
 			\item
 			$g_t \in L^2(0,T; L^2(\Gamma_1) ) $
 			\item
 			$ \Delta  u_0 \in V $ \text{ and } $$\dis  \frac{\partial\,u_0}{\partial\,\nu}  -  g(0) = -(\la+i\,\al) \Delta u_0.  $$
 			\end{multicols}
 		\end{enumerate} Then, the following estimate is satisfied:
 		\begin{equation}
 		\label{teobridge1}\|u_t\|_{C([0,T]; V)} \leq C \Big[ \|f\|_{W^{1,1}(0,T; V) } + \|g\|_{H^1(0,T; L^2(\Gamma_1))} + \| \Delta u_0\|_V   + \|u_0\|_V  \Big].
 		\end{equation}
 		If, in addition, $g\in C([0, T]; H^{1/2}(\Gamma_1)), $ then $ u \in C( [0,T]; H^2(\Omega)). $
 	\end{theo}
 	\begin{proof}
 		Regarding this proof, we shall make use of the formula \eqref{Semigroup}. First, we notice that the conditions imposed on the initial data are equivalent to saying that $ u_0 - \mathcal{N} g(0) \in D(A) .$  Indeed,  this follows
 		from noticing that $ u_0 - \mathcal{N} g(0) \in D(A) $ translates into the following conditions:
 		\begin{equation}\label{comp}
 		\begin{cases}
 		\hspace*{.2cm}\dis \frac{\partial\,(u_0-\mathcal{N}g(0))}{\partial\,\nu}   &= -(\la+i\,\al) \Delta (u_0-\mathcal{N}g(0))\\
        \Rightarrow \displaystyle\frac{\partial\,u_0}{\partial\,\nu} - g(0)   &= -(\la+i\,\al) \Delta u_0\\
 		\Delta (u_0 - \mathcal{N} g(0) )  &\equiv \Delta u_0 \in V. \end{cases}
 		\end{equation}
 		We differentiate the solutions in the sense of duality,  and the calculation is therefore performed in dual spaces.

 		On the other hand,  due to \cite[Theorem 2.4(b), pg.5]{Pazy}, it follows that
 		\begin{eqnarray}
 		\label{f(t)}A\int_0^t e^{A\,\tau}\,{f}(t-\tau)\,d\tau&=&e^{A\,t}\,f(0)-{f}(t).\\
 		\label{Ng(t)}A\int_0^t e^{A\,\tau}\,{\mathcal{N}}g(t-\tau)\,d\tau&=&e^{A\,t}\,\mathcal{N}g(0)-{\mathcal{N}}g(t).
 		\end{eqnarray}
 		Now, employing the Leibniz integral rule,  taking into account  \eqref{f(t)} and \eqref{Ng(t)}, we obtain
 		\begin{equation}
 		\label{der1}\begin{aligned}
 		\frac{d}{dt}\int_0^t e^{A(t-s)}\,{f}(s)\,ds&=\int_0^t \frac{d}{dt}\left[e^{A(t-s)}\,{f}(s)\right]\,ds+{f}(t)\\
 		&=A\int_0^t e^{A(t-s)}\,{f}(s)\,ds+\int_0^t e^{A(t-s)}\,{f}_t(s)\,ds+{f}(t)\\
 		&=A\int_0^t e^{A\,\tau}\,{f}(t-\tau)\,d\tau+\int_0^t e^{A(t-s)}\,{f}_t(s)\,ds+{f}(t)\\
 		&=e^{At}{f}(0)-\cancel{{f}(t)}+\int_0^t e^{A(t-s)}\,{f}_t(s)\,ds+\cancel{{f}(t)}
 		\end{aligned}
 		\end{equation} and
 		{\small\begin{equation}
 			\begin{aligned}\label{der2}
 			-\frac{d}{dt}A\int_0^t e^{A(t-s)}\,{\mathcal{N}}g(s)\,ds&=-A\left[\int_0^t \frac{d}{dt}\left[e^{A(t-s)}\,{\mathcal{N}}g(s)\right]\,ds+{\mathcal{N}}g(t)\right]\\&=-A\left[A\int_0^t e^{A(t-s)}\,{\mathcal{N}}g(s)\,ds+\int_0^t e^{A(t-s)}\,{\mathcal{N}}g_t(s)\,ds+{\mathcal{N}}g(t)\right]
 			\\&=-A\left[A\int_0^t e^{A\,\tau}\,{\mathcal{N}}g(t-\tau)\,d\tau+\int_0^t e^{A(t-s)}\,{\mathcal{N}}g_t(s)\,ds+{\mathcal{N}}g(t)\right]\\&=-A\left[e^{At}{\mathcal{N}}g(0)-\cancel{{\mathcal{N}}g(t)}+\int_0^t e^{A(t-s)}\,{\mathcal{N}}g_t(s)\,ds+\cancel{\mathcal{N}g(t)}\right]\,.
 			\end{aligned}
 			\end{equation}} Taking the derivate of \eqref{Semigroup},  using \eqref{der1} and \eqref{der2}, we obtain the following identity:
 		{\small\begin{equation}\label{yt}\begin{aligned}
 			u_t(t) &=A e^{At} u_0-\frac{d}{dt}A\int_0^t e^{A(t-s)}\,{\mathcal{N}}g(s)\,ds+\frac{d}{dt}\int_0^t e^{A(t-s)}\,{f}(s)\,ds\\&= A e^{At} u_0 + e^{At} {f}(0) + \int_0^t e^{A(t-s)} {f}_t(s) ds  \\ &- Ae^{At} \mathcal{N } g(0)   - A \int_0^t e^{A(t-s)}{\mathcal{N}} g_t (s) ds \\&= A e^{At} [ u_0 - \mathcal{N }g(0) ] + e^{At} {f}(0)   \\&+  \int_0^t e^{A(t-s)} {f}_t(s) ds    - A \int_0^t e^{A(t-s)}{\mathcal{N}} g_t (s) ds  \\
 			&= A e^{At} [ u_0 - {\mathcal{N}} g(0) ] + e^{At} {f}(0) +  \int_0^t e^{A(t-s)} {f}_t(s) ds - \mathcal{L}g_t(t)  \\ &= I + II + III + IV,
 			\end{aligned}
 			\end{equation}} where $ \mathcal{L} $ is the map given in  \eqref{Lmap}.\\
 		
 		Assuming that $u_0 - \mathcal{N} g(0) \in D(A) $, we infer that the term $ I $ belongs  to $C([0,T];V). $ Now, since $f\in W^{1,1} (0,T; V) \hookrightarrow C([0,T]; V)  $, we also obtain that $II \in C([0,T]; V) $. Similarly,  $III \in C([0,T];V) $ by  a standard semigroup argument.

 		 Regarding the last term, we recall the regularity property of the map $ \mathcal{L} $ stated in \eqref{Lmap}. Therefore, $ u_t\in\,C([0,T];V). $ From this  and the fact that $ f\in\,C([0,T];V), $ we obtain that $ \Delta\,u\in\,C([0,T];V). $ In particular, $\Delta u|_{\Gamma_1} \in C([0,T];H^{1/2}(\Gamma_1) ) $.

Now suppose $g$ satisfies the additional assumption $g\in C([0,T]; H^{1/2}(\Gamma_1)). $ Then, since $ \dis \frac{\partial\,u}{\partial\,\nu}=-(\la+i\,\al)\Delta\,u +g, $ we conclude that $\dis \frac{\partial\,u}{\partial\,\nu} \in C([0,T]; H^{1/2}(\Gamma_1) )$. Hence, by elliptic regularity we obtain that $u\in C([0,T]; H^2(\Omega))  $.\\
 	
 		Now, thanks to the compatibility  condition described in \eqref{comp}, and due to \cite[Theorem 2.4(c)]{Pazy}, we have
 		\begin{equation}\label{comp_initial}
 		A\,e^{At} [ u_0 - \mathcal{N} g(0) ]=e^{At}\,A  [ u_0 - \mathcal{N} g(0) ]=(\la+i\,\al)e^{At}\,\Delta\,u_0.
 		\end{equation}
 		
 		It is now an opportune moment to apply the same scenario for $ \mathcal{L} $ in \eqref{Lmap} to $ \int_0^t e^{A(t-s)} {f}_t(s) ds$ in order to obtain an estimate for $ \|v_t\|_{C([0,T];V)} $. In other words, the map $ \mathcal{K} $ given by ${f}_t(t) \mapsto \int_0^t e^{A(t-s)} {f}_t(s) ds $  is continuous from $L^1(0,T; V)$ into $ C([0,T];V), $ keeping in mind that assumption (i) holds true.  Combining this, \eqref{comp_initial}, the continuity of the map $ \mathcal{L} $, and \eqref{yt}, it follows that{\small\begin{equation}\label{vt}
 			\begin{aligned}
 			\|u_t\|_{C([0,T];V)} &\leq C\left\{\|\Delta\,u_0\|_V +\|f\|_{C([0,T];V)}+\|\mathcal{K}f_t\|_{C([0,T];V)}+\|\mathcal{L}g_t\|_{C([0,T];V)}\right\}\\
 			&\leq C\Big[\|u_0\|_V+\|\Delta\,u_0\|_V+\|f\|_{C([0,T];V)}+\|f_t\|_{L^1(0,T; V)}+\|{g}_t\|_{L^2(0,T; L^2(\Ga_1))}\Big]\\
 			&\leq C\Big[\|u_0\|_V+\|\Delta\,u_0\|_V+\|f\|_{W^{1,1}(0,T; V)}+\|{g}\|_{H^1(0,T; L^2(\Ga_1))}\Big],
 			\end{aligned}
 			\end{equation}} where $ C=C(\al, \la,\,T). $

 	\end{proof}
 	\begin{rem}
 		The function $u$ obtained above in the class $C([0,\infty); H^2(\Omega)) $  is a solution to the  problem \eqref{LSf}, that is,
 		\begin{equation}\label{eqbridge}
 		\begin{cases}
 		(\la+i\,\al)\Delta\,u= u_t+\,{f}\quad&\text{ in } \Omega\,\times\,[0,T]\\
 		\dis \frac{\partial\,u}{\partial\,\nu}=-(\la+i\,\al)\,\Delta\,u+{g}\quad&\text{ on } \Gamma_1\,\times\,[0,T]
 		\end{cases}
 		\end{equation} so that $ f\in\,V\,\hookrightarrow\,L^2(\Om)$ and $g\in\,H^{1/2}(\Ga_1). $ So, denoting ~ $ C=C(\al, \la, \,T), $ from the continuity of the trace map $ \gamma_0: H^1(\Om) \rightarrow\,H^{1/2}(\Ga_1) $ and \eqref{eqbridge}, we obtain:
 		{\small	\begin{equation}\label{remH2space}
 			\begin{aligned}
 			\|u(t)\|_{H^2(\Om)}&\leq C\left(\|u_t\|_{L^2(\Om)}+\|{f}\|_{L^2(\Om)}+\|{g}\|_{H^{1/2}(\Ga_1)}+|\lambda+i\,\alpha|\,\|\Delta\,u\|_{H^{1/2}(\Ga_1)}\right)\\
 			&\leq C\left(\|u_t\|_{V}+\|f\|_{V}+\|g\|_{H^{1/2}(\Ga_1)}+\|\Delta\,u\|_{V}\right)\\
 			&\leq C\left(\|u_t\|_{V}+\|f\|_{V}+\|g\|_{H^{1/2}(\Ga_1)}\right)\,.
 			\end{aligned}
 			\end{equation}}
 		
 		Combining Lemma \eqref{l:sem}, \eqref{vt} and \eqref{remH2space},  we get
 	\begin{align}\label{bridge_v}
 	\|u_t\|_{C([0,T]; V)} + \|u\|_{C([0,T]; H^2(\Omega)) }  &\leq C [ \|f\|_{H^{1}(0,T; V) }+\|g\|_{H^1(0,T; H^{1/2}(\Ga_1))}\\ &+\nonumber \| \Delta u_0\|_V   + \|u_0\|_V  ].
 	\end{align}

 \end{rem}
 	The above estimate applied to idbvp yields the following:
 	\begin{theo}\label{t:reg}
 		With  reference to (\ref{LSDynf}),   in addition to the assumptions in Corollary \ref{corollary_inho} assume also that:
 		\begin{enumerate}[(i)] 				
 			\item
 			$f \in H^1(0,T;V) $
 			\item
 			$ \Delta  u_0 \in V $ and $ \dis \frac{\partial\,u_0}{\partial\,\nu}   -  f|_{\Gamma_1} (0) = -(\la+\al\,i) \Delta u_0.   $
 		\end{enumerate}
 		Then
 		$u_t \in C([0,T], V ) $ with appropriate control of the estimates, that is,
 		$$ \|u_t\|_{C([0,T]; V)} + \|u\|_{C([0,T]; H^2(\Omega)) }  \leq C [ \|f\|_{H^{1}(0,T; V) } + \| \Delta u_0\|_V   + \|u_0\|_V  ]. $$
 	\end{theo}
 	\begin{proof}
 		It suffices to apply the previous result with $g \equiv f|_{\Gamma_1}$. Since $f\in H^1(0,T;V) $, by trace theorem we have that $ f|_{\Gamma_1} \in H^1(0,T;H^{1/2}(\Gamma_1))
 		\hookrightarrow H^1(0,T; V ) $. From \eqref{bridge_v}, the desired inequality follows.\\
 		 	\end{proof}

\section{Nonlinear Perturbations}
\setcounter{equation}{0}
 	Our aim  in this section is to prove Theorem \ref{D(A_f)} and Corollary \ref{C:2}.  We will construct solutions for the nonlinear  model \eqref{NLSLip}, where $g(z)$ satisfies the conditions given in Assumption \ref{Assumption_g}.  Moreover, the  function  $f$ is assumed to satisfy \eqref{f}.  As in previous sections,  the well-posedness is obtained by converting this idbvp into a Wentzell ibvp. Namely, we replace $g(u_t)$ on the boundary with $g((\lambda+i\,\alpha)\Delta u + h(u)),$ where the function  $h$ satisfies \eqref{h}. Here, we consider the operator $ A_f $ given in \eqref{operator_Af} with its domain characterized by \eqref{D(A_f)}. 		
 			First, we prove the $\omega-$ maximal dissipativity of the operator $A_f$:\\

 		\noindent{\bf Dissipativity:}
 			
 			Since Since the operator $A_f$ given in \eqref{operator_Af} is nonlinear we will have to take the difference of two solutions. First, we observe by Green's theorem that
 			{\begin{equation}\label{Green} \begin{aligned}
 				(A_fu,v)_V &=   (\lambda+i\,\alpha)(\nabla \cdot\Delta u, \nabla v)_{L^2(\Omega)} + (f(u),v)_V\\
 				&=  -(\lambda+i\,\alpha)(\Delta u,\Delta v)_{L^2(\Omega)} + (\lambda+i\,\alpha)\ds\left(\Delta u, \frac{\partial v}{\partial \nu}\right)_{L^2(\Gamma_1)}\\
 				&+ (f(u),v)_V\\
 				&=  -(\lambda+i\,\alpha)(\Delta u,\Delta v)_{L^2(\Omega)} + \ds\left(g^{-1}\left(-\frac{\partial u}{\partial \nu}\right), \frac{\partial v}{\partial \nu}\right)_{L^2(\Gamma_1)}\\
 				&- \ds\left(h(u), \frac{\partial v}{\partial \nu}\right)_{L^2(\Gamma_1)}+ (f(u),v)_V.
 				\end{aligned}\end{equation}}
 				Hence if we consider the difference between two solutions $u,v\in V$, taking real
parts, recalling Assumption \ref{Assumption_g}:
 				\begin{equation}\label{LipDis}\begin{aligned}
 				(A_f u - A_f v, u-v)_V &\le -(\lambda+i\,\alpha)\|\Delta u-\Delta v\|^2_{L^2(\Omega)} -m\left\|\frac{\partial u}{\partial \nu} - \frac{\partial v}{\partial \nu}  \right\|^2_{L^2(\Gamma_1)}
 				\\ &-  \left(h(u)-h(v), \frac{\partial u}{\partial \nu}-\frac{\partial v}{\partial \nu}\right)_{L^2(\Gamma_1)}\\
 				&+  (f(u)-f(v),u-v)_V .
 				\end{aligned}\end{equation}
 				Using the Cauchy-Schwarz inequality, we get
 				\begin{equation}\label{dis0} \begin{aligned}
 				\text{Re}\left(h(u)-h(v), \frac{\partial u}{\partial \nu}-\frac{\partial v}{\partial \nu}\right)_{L^2(\Gamma_1)} &\leq \| h(u)-h(v)\|_{L^2(\Gamma_1)} \left\|\frac{\partial u}{\partial \nu}-\frac{\partial v}{\partial \nu}\right\|_{L^2(\Gamma_1)}
 				\end{aligned}
 				\end{equation}
 				and
 				\begin{equation*} \text{Re}(f(u)-f(v),u-v)_V \leq \| f(u)-f(v)\|_V \| u-v\|_V .\end{equation*}
 				At this point, we should emphasize that the Lipschitz continuity of $h$ and $f$  plays an essential role here. Since $\|h(u)-h(v)\|_{H^{1/2}(\Gamma_1)} \leq K\| u-v\|_{V}$, we get
 				\begin{equation}\label{dis1} \| h(u)-h(v)\|_{L^2(\Gamma_1)} \left\|\frac{\partial u}{\partial \nu}-\frac{\partial v}{\partial \nu}\right\|_{L^2(\Gamma_1)} \leq K\| u-v\|_V \left\|\frac{\partial u}{\partial \nu}-\frac{\partial v}{\partial \nu}\right\|_{L^2(\Gamma_1)}, \end{equation}
 				and since $\|f(u)-f(v)\|_{V} \leq L\| u-v\|_{V}$, we have
 				\begin{equation}\label{dis2} \| f(u)-f(v)\|_V \| u-v\|_V \leq L\| u-v\|^2_V .\end{equation}

The right hand side of \eqref{dis1} can be estimated by
 				\begin{equation}\label{dis4} K\,\| u-v\|_V \left\|\frac{\partial u}{\partial \nu}
 				-\frac{\partial v}{\partial \nu}\right\|_{L^2(\Gamma_1)} \leq \ds\eta\, K^2 \| u-v\|^2_V
 				+ \ds\frac{1}{4\,\eta} \left\|\frac{\partial u}{\partial \nu}-\frac{\partial v}{\partial \nu}\right\|^2_{L^2(\Gamma_1)} .\end{equation}
 				Combining (\ref{dis2}) and (\ref{dis4}) with (\ref{LipDis}),
 				\begin{equation}\label{LipDis2}\begin{aligned}
 				\text{Re} (A_f u - A_f v, u-v)_V&\leq -\lambda\,\|\Delta u-\Delta v\|^2_{L^2(\Omega)} +  \left(\frac{1}{4\,\eta}-m\right)\left\|\frac{\partial u}{\partial \nu} - \frac{\partial v}{\partial \nu}\right\|^2_{L^2(\Gamma_1)}
 				\\&+ [{\eta} K^2+L]\| u-v\|^2_V.
 				\end{aligned}\end{equation}
 				Now, since $\lambda>0,$ by taking $\omega > {\eta} K^2+L$ with $ \eta $ large enough, we may conclude that
 				\begin{equation} \text{Re} (A_f u - A_f v - \omega I(u-v), u-v)_V \le 0,\end{equation}   which proves that the operator $A$ is $\omega$ -- dissipative.\\

\noindent{\bf Maximality:}
 				
At this point, we consider the space  $Z$ given in \eqref{space Z}. Now, we  define
 			\begin{eqnarray}\label{cont} a(u,v) = \theta\, (u,v)_V &+ & (\lambda+i\,\alpha)(\Delta u, \Delta v)_{L^2(\Omega)} + \left(g^{-1}\left(\frac{\partial u}{\partial \nu}\right), \frac{\partial v}{\partial \nu}\right)_{L^2(\Gamma_1)} \\\notag
 				&- & (f(u),v)_V + \left(h(u), \frac{\partial v}{\partial \nu}\right)_{L^2(\Gamma_1)}. \end{eqnarray}

 				We shall show  that this form is continuous and coercive so that the Browder-Minty theorem can be applied. This will imply that, for every $j\in V\subset Z^\prime$, there exists a unique $u\in Z$ satisfying
 				\begin{equation*} a(u,v) = (-j,v)_V  \text{ for all } v\in Z, \end{equation*}
 				for some value of $\theta$ such that $\text{Re}(\theta)$ is sufficiently large.\\
 				
 Using the triangle inequality, and the bounds on $f, g$, and $h$, we infer
 				\begin{eqnarray}\label{cont2}  |a(u,v)|& \leq& |\theta\, (u,v)_V| +  (\lambda^2+\alpha^2)|(\Delta u, \Delta v)_{L^2(\Omega)}| \\
 				&+& M\left\|\frac{\partial u}{\partial \nu}\right\|_{L^2(\Gamma_1)}^2 \nonumber+ L\|u\|_V\| v\|_V + K\| u\|_V \left\| \frac{\partial v}{\partial \nu}\right\|_{L^2(\Gamma_1)}, \end{eqnarray}
 				for which there exists a bound $C(\lambda, \theta, \alpha,  M, L, K)$ such that
 				\begin{equation} |a(u,v)| \leq C(\lambda, \theta, \alpha, M, L, K) \|u\|_Z \| v\|_Z .\end{equation}
 				\\
 				For coercivity, observe that
 				{\small \begin{align}\label{coer1} a(u,u) &= \theta\,\| u\|^2_V +  (\lambda+i\,\alpha)\|\Delta u\|^2_{L^2(\Omega)} +  \left(g^{-1}\left(\frac{\partial u}{\partial \nu}\right), \frac{\partial u}{\partial \nu}\right)_{L^2(\Gamma_1)} \\
 				&-  (f(u),u)_V - \left(h(u), \frac{\partial u}{\partial \nu}\right)_{L^2(\Gamma_1)}.\nonumber\end{align}}
 				Now, for any complex number  $z= x+iy$,  the inequality $\sqrt{2}|z| \geq |x| + |y|$  can be applied.
 				Moreover, employing Assumption \ref{Assumption_g}, and considering $\text{Im}(\theta) \geq 0$, we obtain
 				{\small\begin{equation}
 					\begin{aligned}\label{coer2}
 					|a(u,u)|\\
&\geq\dfrac{\sqrt{2}}{2}\,\Biggl|\Re(\theta)\|u\|_V^2+\Re\left(g^{-1}\!\left(\frac{\partial u}{\partial \nu}\right), \frac{\partial u}{\partial \nu}\right)_{L^2(\Gamma_1)}- \Re(f(u),u)_V\Biggr.\\
&\Biggl.+\lambda\|\Delta\,u\|_{L^2(\Om)}^2- \Re\left(h(u), \frac{\partial u}{\partial \nu}\right)_{L^2(\Gamma_1)}\Biggr|\\
 					&+\dfrac{\sqrt{2}}{2}\,\Biggl|\Im(\theta)\|u\|_V^2+\alpha\,\|\Delta\,u\|_{L^2(\Om)}^2- \Im(f(u),u)_V \Biggr.-\Biggl. \Im\left(h(u), \frac{\partial u}{\partial \nu}\right)_{L^2(\Gamma_1)}\Biggr|\\
 					&\geq\hspace*{0.5cm}\frac{\sqrt{2}}{2}\left(\Re(\theta)\|u\|_V^2+(\lambda+\alpha)\|\Delta\,u\|_{L^2(\Om)}^2
 					+m\,\left\|\frac{\partial u}{\partial \nu}\right \|_{L^2(\Gamma_1)}^2\right)\\&
 					+\frac{\sqrt{2}}{2}\underbrace{\Im (\theta)}_{\geq 0}\,\|u\|_V^2-\sqrt{2}\left|\left(f(u),u\right)_{V}\right|
 					- \sqrt{2}\left|\left(h(u),\frac{\partial u}{\partial \nu}\right)_{L^2(\Ga_1)}\right|\\
 					&\geq\frac{\sqrt{2}}{2}\left(\Re(\theta)\|u\|_V^2+(\lambda+\alpha)\|\Delta\,u\|_{L^2(\Om)}^2
 					+m\,\left\|\frac{\partial u}{\partial \nu}\right\|_{L^2(\Gamma_1)}^2\right)-\sqrt{2}\,\|f(u)\|_V\,\|u\|_V\\
 					&- \sqrt{2}\,\|h(u)\|_{L^2(\Gamma_1)}\, \left\|\frac{\partial u}{\partial \nu}\right\|_{L^2(\Gamma_1)}\,.\end{aligned}
 					\end{equation}
 					}
 					
 					Making use of  the estimates (\ref{dis2}) and (\ref{dis4})
 					obtained by using the the Lipschitz boundness  of $f$ and $h$, continuity of trace from $V$ into $L^2(\Gamma_1)$, and the Young's inequality for $ \eta>0$ large enough,
 					we get the following estimate:
 					{\small\begin{align}\label{coer3}|a(u,u)| &\geq \frac{\sqrt{2}}{2}\,\text{Re}(\theta)\|u\|^2_V +
 						\frac{\sqrt{2}}{2}(\lambda+\alpha)\,\|\Delta u\|^2_{L^2(\Omega)} +
\left( 						\frac{\sqrt{2}\,m}{2}-\frac{1}{4\,\eta}\right)\ds\left\|\frac{\partial u}{\partial \nu}\right\|_{L^2(\Gamma_1)}^2\\
 						&-\nonumber {\sqrt{2}}\,\left[L + \eta\,K^2\right] \|u\|^2_V
 						\\
 						&=\frac{\sqrt{2}}{2}\,\left[\text{Re}(\theta)-2L - 2\eta\,K^2\right]\|u\|^2_V\nonumber+(\lambda+\alpha) \|\Delta u\|^2_{L^2(\Omega)} + \left( 						\frac{\sqrt{2}\,m}{2}-\frac{1}{4\,\eta}\right)\ds\left\|\frac{\partial u}{\partial \nu}\right\|_{L^2(\Gamma_1)}^2\nonumber
 						\\&\nonumber\geq  C \|u\|^2_Z \end{align}}
 						for some constant $C > 0$ as long as  $ \text{Re}(\theta) > 2L + 2\eta\,K^2$ and $ \eta $ is large enough. \\
 						
 						So, recalling the Browder -- Minty Theorem, if  $\omega > 2L + 2\eta K^2$,
 						the operator $A_f - \omega I$  will be maximal dissipative. From this fact, by the  Lumer -- Philips theorem, the operator $A_f$ generates a strongly continuous semigroup, and therefore Theorem \ref{D(Af)} and Corollary \ref{C:2} are proved.
				
 	   		\section{Local Well-posedness of Strong Solutions  }
 	   		\setcounter{equation}{0}
The main goal of this section is to prove the local existence of solutions (Theorem \ref{T1}) for the problem \eqref{GL_equation_1} at $ H^2 $-level for $ N\le 3. $    We have proved that the linear model with a forcing function $f:\Omega\times (0,T)\to V$ given in \eqref{LSf}
is well-posed in $V$ with appropriate control estimates of the solution map given in Theorem \ref{t:reg}.

We set \begin{equation}\label{F(u)}
F(u) = -(\kappa+i\,\beta)|u|^{p-1}\,u+\ga\,u\,.
\end{equation} To acquire the estimates given in Theorem \ref{t:reg}, we differentiate equation (\ref{GL_equation_1}) in time in the distributional sense. In fact,  let $w = u_t$, then
   	\begin{equation}\label{weq}\begin{cases} w_t-(\la+i\al)\Delta w =F_t  (u,w)& \text{ in }\Omega, \\ w = 0 & \text{ on } \Gamma_0, \\ \ds\frac{\partial w}{\partial \nu} + w_t =0 & \text{ on }\Gamma_1, \\
   	w(0) = w_0 & \text{ in }~ \Omega, \end{cases}\end{equation}
   	where
   	\begin{eqnarray} \label{Ft} F_t  (u,w) &\equiv&   -(\kappa+i\beta)\left\{{\frac{(p+1)}{2}}|u|^{p-1}w+{\frac{(p-1)}{2}}|u|^{p-3}u^2\bar{w}\,\right\}+\gamma\,w  \notag,  \\
   	w_0 &\equiv&   (\la+i\al) \Delta\,u_0 -(\kappa+i\beta)|u_0|^{p-1}u_0+\ga\,u_0 .
   	\end{eqnarray}

   	We will be looking at a (desirably unique) fixed point of the map
   	$$ K : C(0,T; H^2(\Omega) \cap V \times V ) \rightarrow C(0,T; H^2(\Omega) \cap V \times V ), $$
   	defined by
   	$$ K(u^\star,w^\star ) = (u,w),$$
   	where
   	$u$ satisfies \eqref{GL_equation_1} with $F(u^\star) $ and $w$ satisfies (\ref{weq}) with $F_t(u^\star, w^\star ) $.
   	Once such a fixed point is found, it is routine to show that $u$ and $u_t $ are strong solutions to (\ref{GL_equation_1}).

   	In order to establish the existence of a suitable fixed point we need a priori estimates.
   	We begin with some preliminary nonlinear estimates which will be useful both for local and global theory.
   	\begin{lem}\label{l:est} Let $F(u)$ and $ F_t(u,w) $ be given by \eqref{F(u)} and (\ref{Ft}), respectively.  Given $u\in H^2(\Omega)\cap V$
   		and  $w \in V $,
        \begin{enumerate}[(i)]

\item 	if $N\le 3$ and $p\ge 1$, then \begin{equation}\label{F(u)H1} \|F(u)\|_{H^1(\Omega)} \leq C(\|u\|^{p-1}_{L^\infty(\Omega)}+1)\|u\|_{H^1(\Omega)}\leq C(\|u\|^{p-1}_{H^2(\Omega)}+1)\|u\|_{H^1(\Omega)}; \end{equation}
   	\item if $N\le 3$ and $p\ge 2$, then	\begin{equation}\label{F(u)H2} \|F(u)\|_{H^2(\Omega)} \leq (C\|u\|^{p-1}_{L^\infty(\Omega)}+1)\|u\|_{H^2(\Omega)}  \leq (C\|u\|^{p-1}_{H^2(\Omega)}+1)\|u\|_{H^2(\Omega)} ;\end{equation}
\item    if $N\le 3$ and $p\ge 2$, then 		\begin{equation}\label{F_t_V} \begin{aligned}
   		\|F_t(u,w)\|_{V} &\leq C\,\|\nabla\,w\|_{L^2(\Omega)}\,\left\{\|u\|_{H^2(\Om)}^{p-1} +1\right\};
   		\end{aligned}\end{equation}
   \item if $N=1$ and $p\ge 2$, then \begin{equation}\|F_t(u,w)\|_{V}\le C\|w\|_V\left(1+\|u\|_V^{p-1}\right);\end{equation}
          \item if $N = 2$ and $p\ge 2$, then 	\begin{equation}\|F_t(u,w)\|_{V} \leq C\|w\|_{V}\left(\|u\|_{H^2(\Om)}^{\theta+\frac{p-2}{2}}\|u\|_V^{1-\theta}\|u\|_{2}^{\frac{p-2}{2}}\right.\left.+\|u\|_{H^2(\Omega)}^{\frac{p-1}{2}}\|u\|_{2}^{\frac{p-1}{2}}+1\right)\,,\end{equation} where $1>\theta>0$ can be chosen small.  Moreover, for $p\in [2,5]$ and small $\theta$, one has \begin{equation}\|F_t(u,w)\|_{V} \leq C\|w\|_{V}\left(\|u\|_{H^2(\Om)}^2+\|u\|_{V}^{\tau}+1\right)\,\end{equation} where $\tau=\tau(p,\theta)>0$;
           \item if $N=3$ and $p\ge 2$, then \begin{equation}\|F_t(u,w)\|_{V} \leq C\|\,w\|_{V}\left(\|u\|_{H^2(\Om)}^{\theta+\frac{3(p-2)}{4}}\|u\|_V^{1-\theta}\|u\|_{2}^{\frac{p-2}{4}}\right.\left.+\|u\|_{H^2(\Omega)}^{\frac{3(p-1)}{4}}\|u\|_{2}^{\frac{p-1}{4}}+1\right)\,,\end{equation}
 where $1>\theta>0$ can be chosen small. Moreover, for $\displaystyle p\in \left[2,\frac{11}{3}\right]$ and small $\theta$, one has \begin{equation}\|F_t(u,w)\|_{V} \leq C\|w\|_{V}\left(\|u\|_{H^2(\Om)}^2+\|u\|_{V}^{\tau}+1\right)\,\end{equation} where $\tau=\tau(p,\theta)>0$.

        \end{enumerate}

   	\end{lem}
   	\begin{proof}
Note that $$\nabla F(u) =   -(\kappa+i\beta)\left\{{\frac{(p+1)}{2}}|u|^{p-1}\nabla u+{\frac{(p-1)}{2}}|u|^{p-3}u^2\nabla\bar{u}\,\right\}+\gamma\,\nabla u.$$  Therefore,
$$|\nabla F(u)|\le C(|u|^{p-1}|\nabla u|+|\nabla u|).$$

   		The inequality (\ref{F(u)H1}) as well as  the second part  of (\ref{F(u)H2}) follow directly from embedding  $H^2(\Omega)\hookrightarrow\,L^\infty(\Omega)$ . The first inequality in (\ref{F(u)H2}) was proved  in \cite{Brezis-Gallouet} in the case of $p=3$ and $N=2$.  A similar proof also applies to the more general case given in this lemma.  The idea is to use the Gagliardo-Nirenberg inequality $$\|u\|_{W^{1,4}(\Omega)}\le \|u\|_{L^\infty}^{\frac{1}{2}}\|u\|_{H^2(\Omega)}^{\frac{1}{2}}.$$
   		
   		In order to prove inequality (\ref{F_t_V}), we first estimate $\nabla F_t(u,w)$:
   		{\small   			\begin{align} |\nabla F_t(u,w)|&\le C\left(|u|^{p-2}|\nabla u||w|+|u|^{p-1}|\nabla w|+|\gamma||\nabla w|\right)
   			\end{align}}
   		Setting $ C:=C(\kappa,\beta,p), $ by the triangle inequality, we have
   		\begin{equation}\label{F_t_V1} \begin{aligned}
   		\|F_t(u,w)\|_{V} &\leq C\left\{\|u\|_{L^\infty(\Om)}^{p-2}\,\|w\nabla u\|_{L^2(\Omega)}+\|u\|_{L^\infty(\Om)}^{p-1}\,\|\nabla\,{w}\|_{L^2(\Omega)}+|\ga|\,\|\nabla\,w\|_{L^2(\Om)}\right\}\,.
   		\end{aligned}\end{equation}  Now, we use $\|w\nabla u \|_{L^2(\Omega)}\le \|w\|_V\|u\|_{H^2(\Omega)}$ and $H^2(\Omega)\hookrightarrow L^\infty(\Omega)$
 to obtain \eqref{F_t_V}.\\

        If $N = 1$, then using $H^1(\Omega)\hookrightarrow L^\infty(\Omega),$ we can easily obtain from \eqref{F_t_V1} that $$\|F_t(u,w)\|_{V}\le C\|w\|_V\left(1+\|u\|_V^{p-1}\right).$$

        If $N = 2$, then  by using the Gagliardo-Nirenberg inequality $$\|\nabla u\|_{L^s(\Omega)}\le\|u\|_{H^2(\Omega)}^{\theta}\|u\|_V^{1-\theta},$$  where $\displaystyle s>2$ and $0<\displaystyle\theta=1-\frac{2}{s}<1,$ and the Sobolev embedding $H^1(\Omega)\hookrightarrow L^q(\Omega) $ ($q\ge 1)$, we obtain
 		\begin{equation}\begin{aligned}
   		\|w\nabla u\|_{L^2(\Omega)}&\leq C\|w\|_{L^r(\Omega)}\,\|\nabla\,u\|_{L^s(\Omega)}\\ &\leq C\|w\|_{V}\|u\|_{H^2(\Omega)}^\theta\|u\|_V^{1-\theta}
   		\end{aligned} \end{equation} where $r=\displaystyle\frac{2s}{s-2}$.  Moreover, for $N=2$, one has the Gagliardo-Nirenberg inequality:
$$\|u\|_{\infty}\le \|u\|_{H^2(\Omega)}^{\frac{1}{2}}\|u\|_{2}^{\frac{1}{2}}.$$

Therefore, from \eqref{F_t_V1} we obtain
	\begin{equation} \begin{aligned}
   		\|F_t(u,w)\|_{V} &\leq C\|w\|_{V}\left\{\|u\|_{H^2(\Om)}^{\theta+\frac{p-2}{2}}\|u\|_V^{1-\theta}\|u\|_{2}^{\frac{p-2}{2}}\right.\left.+\|u\|_{H^2(\Omega)}^{\frac{p-1}{2}}\|u\|_{2}^{\frac{p-1}{2}}+1\right\}\,.
   		\end{aligned}\end{equation}
Observe that in the above inequality, $\frac{p-1}{2}\le 2$ if $p\le 5.$

  If $N = 3$,we can again use the Gagliardo-Nirenberg inequality in the form $$\|\nabla u\|_{L^3(\Omega)}\le\|u\|_{H^2(\Omega)}^{\frac{1}{3}}\|u\|_V^{\frac{2}{3}}$$ and use the embedding $H^1(\Omega)\hookrightarrow L^6(\Omega)$ to obtain
 		\begin{equation}\begin{aligned}
   		\|w\nabla u\|_{L^2(\Omega)}&\leq C\|w\|_{L^6(\Omega)}\,\|\nabla\,u\|_{L^3(\Omega)}\\ &\leq C\|w\|_{V}\|u\|_{H^2(\Omega)}^\theta\|u\|_V^{1-\theta}.
   		\end{aligned} \end{equation}  Moreover, for $N=3$, one has the Gagliardo-Nirenberg inequality:
$$\|u\|_{\infty}\le \|u\|_{H^2(\Omega)}^{\frac{3}{4}}\|u\|_{2}^{\frac{1}{4}}.$$

Therefore, we obtain
	\begin{equation} \begin{aligned}
   		\|F_t(u,w)\|_{V} &\leq C\|\,w\|_{V}\left\{\|u\|_{H^2(\Om)}^{\theta+\frac{3(p-2)}{4}}\|u\|_V^{1-\theta}\|u\|_{2}^{\frac{p-2}{4}}\right.\left.+\|u\|_{H^2(\Omega)}^{\frac{3(p-1)}{4}}\|u\|_{2}^{\frac{p-1}{4}}+1\right\}\,.
   		\end{aligned}\end{equation}
Note that in the above case, $\frac{3(p-1)}{4}\le 2$ if $p\le \frac{11}{3}.$
   	\end{proof}

Now, we take into account the following compatibility condition:
 	   		\begin{defin}{[Compatibility Condition (CC)]} \label{CC} We consider
 	   			$$ \frac{\partial\,u_0}{\partial\,\nu}  +  (\la+i\al) \Delta\,u_0 -F(u_0) =0 ~\text{on} ~ \Gamma_1\,.$$
 	   		\end{defin}

 	   		We also define the following spaces:
 	   		\begin{equation*} X_0 =\begin{cases}
 	   		&(u_0,w_0)\in   V  \times  V\\
 	   		& w_0 =  (\la+i\al) \Delta\,u_0 -F(u_0)\\
 	   		&  \Delta u_0 \in V\\
 	   		& u_0   \text{  satisfies CC (Definition \ref{CC})}
 	   		\end{cases} \end{equation*} and the Banach space
 	   		\begin{equation*} X_T = \left\{(u,w): u\in C[0,T; H^2(\Omega)\cap H^1_{\Gamma_0}(\Omega)], w\in C[0,T; H^1_{\Gamma_0}(\Omega)), u_t = w\right\}\,.\end{equation*}

 	   		We note that from the elliptic theory, we get $$\Delta u_0 \in  V \hookrightarrow H^1(\Omega)  ~\text{and}  ~ \frac{\partial\,u_0}{\partial\,\nu} = - w_0|_{\Gamma_1} \in H^{1/2}(\Gamma). $$ So, we have  the implication: $ (u_0,w_0 ) \in X_0 \Rightarrow u_0 \in H^2(\Omega) $.
 	   		
 	   		\medskip

 	   		Thus, it makes sense to define  the following norms on  $X_0$ and $ X_T $ :
 	   		\begin{equation*} \|(u,w)\|^2_{X_0} = \|u\|^2_{H^2(\Omega)} + \|w\|^2_{V}, \end{equation*}
 	   		\begin{equation*} \|(u,w)\|^2_{X_T} = \ds\sup_{t\in [0,T]} \|u\|^2_{H^2(\Omega)} + \ds\sup_{t\in [0,T]} \|w\|^2_{V} .\end{equation*}
 	   		
 	   		We will prove Theorem \ref{T1} in several steps.\\
 	   		
 	   		\smallskip

 	   			{\bf Step 1: Setting the map. }  We will restrict the map $K$ defined previously to a special complete metric space in order to obtain a contraction.
 \begin{lem}\label{QTspace}
 	   		Let $$Q_T\equiv \{(u^*,w^*)\in X_T \text{ s.t. } u^*(0)=u_0\}.$$  Then,  $Q_T$ is a non-empty complete metric space with the metric induced from the norm of $X_T$, i.e., with the metric given by $$d_{Q_T}((u_1^*,w_1^*),(u_2^*,w_2^*))=\ds\sup_{t\in [0,T]} \|u_1^*-u_2^*\|^2_{H^2(\Omega)} + \ds\sup_{t\in [0,T]} \|w_1^*-w_2^*\|^2_{H^1_{\Gamma_0}(\Omega)}.$$
 \end{lem}
 \begin{proof} $(u_0,0)\in Q_T$, hence $Q_T$ is nonempty.  It is easy to see that $d_{Q_T}$ is a metric.  Now, in order to show completeness of $Q_T$, take $(u_n^*,w_n^*)\in Q_T$ such that $(u_n^*,w_n^*)\rightarrow (u^*,w^*)\in X_T$.  This means $u_n^*(0)=u_0$ and $\displaystyle\lim_{n}\left[\sup_{t\in [0,T]} \|u_n^*-u^*\|^2_{H^2(\Omega)}\right]=0,$  which implies $$0\le \|u^*(0)-u_0\|_{H^2(\Omega)}= \|u^*(0)-u_n^*(0)\|_{H^2(\Omega)}\le \sup_{[0,T]}\|u^*-u_n^*\|_{H^2(\Omega)}.$$  Passing to the limit as $n\rightarrow \infty$, we get $u^*(0)=u_0$.  Namely, $(u^*,w^*)\in Q_T.$  That is $Q_T$ is closed.  But closed subsets of complete spaces are complete.  Thereofore, $Q_T$ is complete since $X_T$ is complete.
 \end{proof}

  Now let $(u_0, w_0) \in X_0.$ We  consider the map $ K(u^\star,w^\star )$ , with $(u^\star, w^\star )\in Q_T  $, which produces  solutions  $(u, w) \in C^1([0,T]; V) \cap C([0, T]; V)
 	   			$ to the respective problems:
 	   			\begin{equation}\label{NLSahat}\begin{cases} u_t-(\lambda+i\alpha)\Delta\,u =F(u^\star), \text { in } \Omega\times \mathbb{R_+},\\
 	   			u=0,  \text { on } \Gamma_0\times \mathbb{R_+},\\
 	   			\dis \frac{\partial u}{\partial \nu} +(\lambda+i\alpha)\Delta\,u-F(u^\star)=0, \text { on } \Gamma_1\times \mathbb{R_+},\\
 	   			u(0)=u_0,  \text { in } \Omega\end{cases}\end{equation}
 	   			and
 	   			\begin{equation}\label{zeqhat}\begin{cases} w_t-(\lambda+i\alpha)\Delta\,w =F_t(u^\star,w^\star), \text { in } \Omega\times \mathbb{R_+},\\
 	   			w=0,  \text { on } \Gamma_0\times \mathbb{R_+},\\
 	   			\dis \frac{\partial w}{\partial \nu} +(\lambda+i\alpha)\Delta\,w-F_t(u^\star,w^\star)=0, \text { on } \Gamma_1\times \mathbb{R_+},\\
 	   			w(0)=w_0,  \text { in } \Omega.\end{cases}\end{equation}

 We would like to give some important remarks at this point.
 \begin{rem}
 \begin{itemize}
   \item[(i)] It is important to note that $u_0$ satisfies the necessary compatibility conditions since $(u^*,w^*)$ is taken from the specially constructed space $Q_T$ introduced in the previous lemma, which enforces the equality $F(u^*(0))=F(u_0)$.
   \item[(ii)] One should also observe that $K$ maps the elements of $Q_T$ again to the elements of $Q_T$.  This follows from the linear theory and the fact that $(u,w)=K(u^*,w^*)$ is continuous on [0,T] so that in particular $u(0)=u_0$, i.e., $(u,w)\in Q_T$.
 \end{itemize}

 \end{rem}
 	   			Componentwise, the map $K(u^\star,w^\star ) $ can be thought as the composition of \\
 	   			\begin{eqnarray*}\begin{cases} K(u^\star,\cdot ):u^\star\longmapsto F(u^\star)\longmapsto u \\
 	   					K(\cdot ,w^\star):w^\star\longmapsto F_t(\cdot,w^\star)\longmapsto w .\end{cases}\end{eqnarray*}
 	   			These component maps can be found explicitly via the ``boundary solver'' introduced in Lemma \ref{l:sem} as follows: \\
 	   			{\small   		\begin{align}\label{Kdef}
 	   				K(u^\star,\cdot )&=u(t)\nonumber\\
 	   				&=\left( e^{At}u_0^*- \ds\int_0^t e^{A(t-s)}A\, {\mathcal{N}}\, F(u^\star(s))\, ds+  \ds\int_0^t e^{A(t-s)}{F}(u^\star(s))\, ds\right) \\
 	   				K(\cdot \nonumber ,w^\star)&=w(t)\\
 	   				&=\,\left( e^{At}w_0^*- \ds\int_0^t e^{A(t-s)}A\, {\mathcal{N}}\, F_t(\cdot, w^\star(s))\, ds + \ds\int_0^t e^{A(t-s)}{F}_t(\cdot,w^\star(s))\, ds\right) .\end{align}} \\

 	   		As before, $A$ is the operator given in \eqref{operator_A} with Wentzell boundary conditions, $\mathcal{N}$ is the Neumann map as seen in (\ref{Neumann}).
 	   		
 	   		{\bf Step 2: The estimates: Invariance of the Ball in $X_T $ }
 	   		
 	   		We recall that $K (u^\star,w^\star ) = (u,w)$ where $u$ satisfies (\ref{NLSahat}) with the right hand side $f \equiv F(u^\star) $ and $w$ satisfies (\ref{zeqhat}) with the right hand side $f\equiv F_t(u^\star,w^\star ) $. Since the initial data satisfies the required compatibility conditions, we are in  a position to apply the estimates of Theorem \ref{t:reg} . This yields:
 	   		\begin{eqnarray}\label{equw1}
 	   		&&\|w \|_{C([0,T]; V)}+ \|u\|_{C([0,T]; H^2(\Om)) }  \\ \notag  &&\leq  C\left(
 	   		\| F(u^\star) \|_{L^2(0,T; V) } + \|F_t(u^\star,w^\star)\|_{L^2(0,T,V)} + \|u_0\|_V + \| \Delta u_0\|_V \right).
 	   		\end{eqnarray}

 	   		We first need to verify that $K(u^\star,w^\star)$ maps $B_R(Q_T)$ into $B_R(Q_T)$, where $B_R(Q_T)$ denotes the closed ball  of radius $R$ in the space $Q_T$. Below we will suitably choose $R$ and $T$.
 	   		To accomplish this, we shall use the estimates in (\ref{equw1}) and Lemma \ref{l:est}.  To this end, let $R>0$ be a fixed number (to be chosen in a moment) such that $(u^*,w^*)\in B_{R}(X_T).$
 	   		
 	
 	   		Making use of estimates \eqref{F(u)H1} and \eqref{F_t_V} given  in Lemma \ref{l:est}, we conclude that
 	   		\begin{equation}
 	   		\label{equw3}
 	   		\begin{aligned}
 	   		&\| F(u^\star) \|_{L^2(0,T; V) } + \|F_t(u^\star,w^\star)\|_{L^2(0,T,V)}\\
 	   		&= \left[\int_0^T \|F(u^\star(t))\|_V^2\,dt\right]^{1/2}+
 	   		\left[\int_0^T \|F_t(u^\star(t), w^\star(t)\|_V^2\,dt\right]^{1/2}\\
 	   		&\leq \left[\int_0^T2\left(1+\|u^\star(t)\|_{H^2(\Om)}^{2(p-1)}\right)\,\left(\|u^\star(t)\|_V^2+\|w^\star(t)\|_V^2\right)\,dt\right]^{1/2}\\
 	   		&\leq C\,T^{1/2}\,\left(1+\|u^\star\|_{C([0,T]; H^2(\Om))}^{p-1}\right)\,\left(\|u^\star\|_{C([0,T]; H^2(\Om))}+\|w^\star\|_{C([0,T]; V)}\right).\\
 &\leq C\,2RT^{1/2}\,\left(1+R^{p-1}\right)\,.
 	   		\end{aligned}
 	   		\end{equation}

 	   		Combining  \eqref{equw1} and \eqref{equw3}, we infer,
 	   		\begin{equation*}
 	   		\begin{aligned}
 	   		\|K(u^\star, w^\star)\|_{X_T}&=\|({u},{w})\|_{X_T}\\
 	   		&=\|w\|_{C([0,T]; V)} + \|u\|_{C([0,T]; H^2(\Om))}\\
 	   		&\le C_{u_0}+C\,2RT^{1/2}\,\left(1+R^{p-1}\right).
 	   		\end{aligned}
 	   		\end{equation*}
 	   		Let $R=2C_{u_0}$.  Then, for small $T$, $K$ maps $B_{R}(Q_T)$ into itself.\\

 	   		{\bf Step 3: Contraction}.
 	   		We will show that for small $T$, there exist  $ 1>\rho >0$  such that
 	   		{\small$$
 	   			\dis \| K(u^\star_1,w^\star_1) - K(u^\star_2, w^\star_2 ) \|_{X_T} \leq \rho\,\|(u^\star_1 -u^\star_2, w^\star_1 -w^\star_2)\|_{X_T},\,\forall\, (u^\star_1, w^\star_1), (u^\star_2, w^\star_2)\in\,B_R(X_T).
 	   			$$}
 	   		Let $(u_1^*,w_1^*),(u_2^*,w_2^*)\in X_T$. Then by arguments similar to those above that, we can obtain
 	   		{\begin{equation}\label{contract1}
 	   			\begin{aligned}
 	   			\| K(u^\star_1,w^\star_1) - K(u^\star_2, w^\star_2 ) \|_{X_T} &= \|(u_1-u_2, w_1-w_2)\|_{X_T}\\
 	   			&=\|u_1-u_2\|_{C([0,T]; H^2(\Om))}+\|w_1-w_2\|_{C([0,T]; V)}\\&\leq\int_0^T\|F(u^\star_1(s))-F(u^\star_2(s))\|_{H^2(\Omega)}\,ds\\
 	   			&+\int_0^T\|F_t(u^\star_1(s), w^\star_1(s))-F_t(u^\star_2(s), w^\star_2(s))\|_{V}\,ds\,.\\
 	   			\end{aligned}
 	   			\end{equation}}
 	   		By using the local Lipschitz estimates for $F$ and $F_t$, taking into account that $ (u^\star_1, w^\star_1), (u^\star_2, w^\star_2)\in\,B_R(X_T), $ one obtains:
 	   		{\small\begin{eqnarray*}
 	   				\|F(u^\star_1) - F(u^\star_2) \|_{H^2(\Om)}\!\!&\leq&\!\! C_1\left(\|u^\star_1\|^{c(p)}_{H^2(\Om)},\, \|u^\star_2\|^{c(p)}_{H^2(\Om)}, |\ga|\right) \|u^\star_1 -u^\star_2\|_{H^2(\Om)}\\\!\!&\leq&\!\! C_1(R)\, \|u^\star_1 -u^\star_2\|_{H^2(\Om)}\,\end{eqnarray*}} and
 	   		\begin{eqnarray*}
 	   			\|F_t(u^\star_1,w^\star_1) - F(u^\star_2,w^\star_2) \|_{V}\!\!&\leq&\!\! C_2\left(\|u^\star_1\|^{c(p)}_{H^2(\Omega)}, \, \|u^\star_2\|^{c(p)}_{H^2(\Omega)},|\ga|\right) \|w^\star_1 -w^\star_2\|_V\\\!\!&\leq&\!\! C_2(R) \, \|w^\star_1 -w^\star_2\|_V.
 	   		\end{eqnarray*}
 	   		Setting $ C_3:=\max\{C_1(R),C_2(R)\}, $ from \eqref{contract1}, we have
 	   		{\small\begin{equation}\label{contract2}
 	   			\begin{aligned}
 	   			&\| K(u^\star_1,w^\star_1) - K(u^\star_2, w^\star_2 ) \|_{X_T}\\
 	   			&\leq C_3\!\!\int_0^T\!\!\left[ \|u^\star_1(s)-u^\star_2(s)\|_{H^2(\Om)}+\|w^\star_1(s)-w^\star_2(s)\|_{V}\right]ds\\
 	   			&\leq T\,C_3 \left[\|u^\star_1-u^\star_2\|_{C([0,T]; H^2(\Om))}+\|w^\star_1-w^\star_2\|_{C([0,T]; V)}\right]\\
 	   			&\le T C_3\, \|(u^\star_1-u^\star_2, w^\star_1-w^\star_2)\|_{X_T}.
 	   			\end{aligned}
 	   			\end{equation}}
 The above estimate proves that $K$ is a contraction if $T$ is sufficiently small.  This proves the local existence and uniqueness of a strong solution in $B_R(Q_T)$.

 \begin{rem}
   The above result says that there is a unique local solution in $B_R(Q_T)$.  This does not mean that there is uniqueness in $Q_T$.  In the case, $\alpha,\beta>0$, one can also prove the uniqueness in $Q_T$.  This directly follows from \eqref{EE0}.
 \end{rem}

\section{Global Well-posedness of Strong Solutions}
	 	 	   		\setcounter{equation}{0} 	   	
In this section, we study the global strong solutions of the CGLE with dynamic boundary conditions with power type nonlinearity under the additional assumption $\beta> 0$. Our proofs for strong solutions will use Sobolev embeddings. Therefore, we will have some restriction on $p$.  In [4], it is proven that the defocusing cubic NLS with dynamic boundary conditions is globally well-posed for $N=2$. Here, we improve this result in the context of the CGLE by proving the well-posedness of global solutions for dimensions $N\le 3$. More precisely, we prove global well-posedness for  $p\ge 2$ if $N=1$; $\displaystyle p\in [2,5]$ if $N=2$; and $\displaystyle p\in \left[2,\frac{11}{3}\right]$ if $N=3$.  The smoothing effect will play a major role in this context. This is a missing ingredient in the case of the nonlinear Schrödinger equations.

	 	 	   		Suppose $ 0<T_{\max}\le\infty $ denotes the maximal time of existence of a given local strong solution. We want to prove that $T_{\max}=\infty$ by showing that $\|(u,u_t)\|_{X_T}$ remain bounded on $ [0,T_{\max})$.  To this end, we first prove the following lemma which is true in any dimension.  Similar estimates were proved, for example, in \cite[Lemmas 3.2 and 5.2]{Bechouche}, where the domain was either the whole space or a torus (periodic boundary conditions).  Here, our aim is to obtain a uniform energy bound for the solutions.

\begin{lem}\label{H1global}
  Let $u_0\in V\cap L^{p+1}(\Omega)$ with $\gamma_0u_0\in L^{p+1}(\Gamma_1)$.  Then,

\begin{enumerate}[(i)]
  \item if $\gamma> 0$, then $E(t)\le E(0)\exp(CT_{max})$ for $t\in [0,T_{max}),$ where

\begin{multline}\label{EtDef}E(t)\equiv \frac{\alpha}{2} \|\nabla u(t)\|_{L^2(\Omega)}^2 +\frac{\beta}{p+1}\|u(t)\|_{L^{p+1}(\Omega)}^{p+1}+\frac{1}{p+1}(\alpha\kappa+\beta\lambda)\|u\|_{L^{p+1}(\Gamma_1)}^{p+1}\\
+\alpha\int_0^t \|u_t(s)\|_{L^2(\Gamma_1)}^2ds+\alpha\lambda\int_0^t\|\Delta u\|_{L^2(\Omega)}^2ds+\kappa\beta\int_0^t\|u(s)\|_{L^{2p}(\Omega)}^{2p}.\end{multline}
  \item if $\gamma\le 0$, then $E(t)\le E(0)$ for $t\in [0,T_{max}),$ where

\begin{multline}\label{EtDef2}E(t)\equiv \frac{\alpha}{2} \|\nabla u(t)\|_{L^2(\Omega)}^2 +\frac{\beta}{p+1}\|u(t)\|_{L^{p+1}(\Omega)}^{p+1} -\frac{\alpha\gamma}{2}\|u(t)\|_{L^2(\Gamma_1)}^2+\frac{1}{p+1}(\alpha\kappa+\beta\lambda)\|u\|_{L^{p+1}(\Gamma_1)}^{p+1}\\
+\alpha\int_0^t \|u_t(s)\|_{L^2(\Gamma_1)}^2ds+\alpha\lambda\int_0^t\|\Delta u\|_{L^2(\Omega)}^2ds+\kappa\beta\int_0^t\|u(s)\|_{L^{2p}(\Omega)}^{2p}.\end{multline}
\end{enumerate}

Moreover, in both cases above, the solutions enjoy the following interior and boundary regularity for all $T>0$:
$$(u,u_t)\in [C([0,T];V\cap L^{p+1}(\Omega))\cap L^2(0,T;H^2(\Omega))\cap L^{2p}(0,T;L^{2p}(\Omega))]\times L^2(0,T;L^2(\Omega)),$$
$$\frac{\partial u}{\partial n}=\gamma_0u_t\in L^2(0,T;L^2(\Gamma_1)), \gamma_0u\in C([0,T];L^{p+1}(\Gamma_1)).$$
\end{lem}
\begin{proof}
 We start by taking the real part of the scalar product of $-\alpha\Delta u+\beta |u|^{p-1}u$ and $u_t$:
\begin{multline}\label{73}
  Re(-\alpha\Delta u+\beta |u|^{p-1}u,u_t) = Re(-\alpha\Delta u+\beta |u|^{p-1}u,(\lambda+i\alpha)\triangle u-(\kappa+i\beta)|u|^{p-1}u+\gamma u)\\
  =-\alpha\lambda\|\Delta u\|_{L^2(\Omega)}^2+\alpha\gamma \|\nabla u\|_{L^2(\Omega)}^2-\alpha\gamma Re(\frac{\partial u}{\partial n},u)_{L^2(\Gamma_1)}-\kappa\beta\|u\|_{L^{2p}(\Omega)}^{2p}+\beta\gamma\|u\|_{L^{p+1}(\Omega)}^{p+1}\\
  +\alpha Re(\Delta u,(\kappa+i\beta)|u|^{p-1}u)+\beta Re(|u|^{p-1}u,(\lambda+i\alpha)\triangle u)\\
  =-\alpha\lambda\|\Delta u\|_{L^2(\Omega)}^2+\alpha\gamma \|\nabla u\|_{L^2(\Omega)}^2-\alpha\gamma Re(\frac{\partial u}{\partial n},u)_{L^2(\Gamma_1)}-\kappa\beta\|u\|_{L^{2p}(\Omega)}^{2p}+\beta\gamma\|u\|_{L^{p+1}(\Omega)}^{p+1}\\
  -\alpha Re(\nabla u,(\kappa+i\beta)\nabla\left(|u|^{p-1}u\right))+\alpha Re(\frac{\partial u}{\partial n},(\kappa+i\beta)|u|^{p-1}u)_{L^2(\Gamma_1)}\\
  -\beta Re(\nabla(|u|^{p-1}u),(\lambda+i\alpha)\nabla u)+\beta Re(|u|^{p-1}u,(\lambda+i\alpha)\frac{\partial u}{\partial n})_{L^2(\Gamma_1)}.
\end{multline}
The first boundary term at the right hand side of \eqref{73} can be estimated by

\begin{multline}
  \left|-\alpha\gamma Re(\frac{\partial u}{\partial n},u)_{L^2(\Gamma_1)}\right|\le \frac{\alpha\gamma}{2}\left(C_\epsilon\|u\|_{L^2(\Gamma_1)}^2+\epsilon\left\|\frac{\partial u}{\partial n}\right\|_{L^{2}(\Gamma_1)}^{2}\right)\\
  \le \frac{\alpha\gamma}{2}\left(C_\epsilon\|\nabla u\|_{L^2(\Omega)}^2+\epsilon\left\|u_t\right\|_{L^{2}(\Gamma_1)}^{2}\right).
\end{multline}

Alternatively, by using $\displaystyle \frac{\partial u}{\partial \nu}|_{\Gamma_1}=-u_t$ one also has

\begin{equation}\label{alternative}
 -\alpha\gamma Re(\frac{\partial u}{\partial n},u)_{L^2(\Gamma_1)}=\frac{\alpha\gamma}{2}\frac{d}{dt}\|u\|_{L^2(\Gamma_1)}^2.
\end{equation}


Under the assumption $\beta>0$, we have $\alpha\kappa+\beta\lambda>0$ and

\begin{multline}
  -\alpha Re(\nabla u,(\kappa+i\beta)\nabla\left(|u|^{p-1}u\right))-\beta Re(\nabla(|u|^{p-1}u),(\lambda+i\alpha)\nabla u)\\
  =-(\alpha\kappa+\beta\lambda)Re(\nabla(|u|^{p-1}u),\nabla u)
  = -(\alpha\kappa+\beta\lambda) Re(\nabla u,(\frac{p+1}{2}|u|^{p-1}\nabla u+\frac{p-1}{2}|u|^{p-3}u^2\nabla \bar{u}))\\
  =-(\alpha\kappa+\beta\lambda)\left[\frac{p+1}{2}\int_\Omega|u|^{p-1}|\nabla u|^2dx+\frac{p-1}{2} Re\int_\Omega|u|^{p-3}\bar{u}^2(\nabla{u})^2dx\right]\le 0.
\end{multline}

The last term at the right hand side of \eqref{73} is calculated as follows:

\begin{multline}
  \alpha Re(\frac{\partial u}{\partial n},(\kappa+i\beta)|u|^{p-1}u)_{L^2(\Gamma_1)} +\beta Re(|u|^{p-1}u,(\lambda+i\alpha)\frac{\partial u}{\partial n})_{L^2(\Gamma_1)}\\
  =(\alpha\kappa+\beta\lambda)Re(|u|^{p-1}u,\frac{\partial u}{\partial n})_{L^2(\Gamma_1)}
  =-\frac{1}{(p+1)}(\alpha\kappa+\beta\lambda)\frac{d}{dt}\|u\|_{L^{p+1}(\Gamma_1)}^{p+1}.
\end{multline}

On the other hand, by using the main equation, we can rewrite the same scalar product in \eqref{73} as

\begin{multline}
  Re(-\alpha\Delta u+\beta |u|^{p-1}u,u_t)=\frac{\alpha}{2} \frac{d}{dt}\|\nabla u\|_{L^2(\Omega)}^2 +\frac{\beta}{p+1}\frac{d}{dt}\|u\|_{L^{p+1}(\Omega)}^{p+1}-\alpha Re(\frac{\partial u}{\partial n},u_t)_{L^2(\Gamma_1)}\\
  =\frac{\alpha}{2} \frac{d}{dt}\|\nabla u\|_{L^2(\Omega)}^2 +\frac{\beta}{p+1}\frac{d}{dt}\|u\|_{L^{p+1}(\Omega)}^{p+1}+\alpha \|u_t\|_{L^2(\Gamma_1)}^2.
\end{multline}
  If $\gamma\ge 0$, it follows that
\begin{multline}\label{energyEst}
 E(t)-E(0)\le \alpha\gamma(1+\frac{C_\epsilon}{2})\int_0^t\|\nabla u(s)\|_{L^2(\Omega)}^2ds+\beta\gamma\int_0^t\|u(s)\|_{L^{p+1}(\Omega)}^{p+1}
  +\epsilon\frac{\alpha\gamma}{2}\int_0^t\|u_t(s)\|_{L^{2}(\Gamma_1)}^{2}ds\\
  \le \epsilon\frac{\gamma}{2}E(t) + C\int_0^tE(s)ds.
\end{multline}

where $\epsilon>0$ small and fixed.

Therefore,

$$E(t)\le CE(0)+C\int_0^tE(s)ds.$$

Now, using Gronwall's inequality, we conclude that \begin{equation}\label{EE0}E(t)\le E(0)\exp(CT_{max})\end{equation} for $t\in [0,T_{max})$.

Now, we consider the case $\gamma<0$ and define $E(t)$ as in \eqref{EtDef2}.  Bu utilizing the alternative calculation in \eqref{alternative}, we get

\begin{equation}\label{energyEst2}
 E(t)-E(0)\le \frac{\alpha\gamma}{2}\int_0^t\|\nabla u(s)\|_{L^2(\Omega)}^2ds+\beta\gamma\int_0^t\|u(s)\|_{L^{p+1}(\Omega)}^{p+1}
  \le 0.
\end{equation}  This gives $E(t)\le E(0)$ when $\ga<0.$

Hence, we proved that the corresponding solution $(u,u_t)$ is global in the sense that, for all $T>0$, $$(u,u_t)\in \left[C([0,T];V\cap L^{p+1}(\Omega))\cap L^2(0,T;H^2(\Omega))\cap L^{2p}(0,T; L^{2p}(\Omega))\right]\times L^2(0,T;L^2(\Omega))$$  and
$$\gamma_0u\in C([0,T];L^{p+1}(\Gamma_1)), \frac{\partial u}{\partial n}=\gamma_0u_t\in L^2(0,T;L^2(\Gamma_1)).$$
\end{proof}

\begin{lem}\label{gagliardo} Let $u$ be a local solution.  Then,
\begin{enumerate}
  \item If $N=1$ and $1< p<\infty$, then $\|F(u)\|_V\le C.$
  \item If $N=2$ and $1< p<\infty$, then $\|F(u)\|_V\le C+C\|u\|_{H^2(\Omega)}^{\theta}$ where $1>\theta>0$ can be chosen as small as we wish.
  \item If $N\ge 3$ and $1<p<\frac{N}{N-2}$, then $\|F(u)\|_V\le C+C\|u\|_{H^2(\Omega)}^{\theta}$ where $1>\theta>0$ can be chosen as small as we wish.
  \item If $N=3$ and $p\ge 3$, then $\displaystyle\|F(u)\|_V\le C+C\|u\|_{H^2(\Omega)}^{\frac{3(p-1)}{4}}$.  If in addition $p<\frac{11}{3}$, then $\|F(u)\|_V\le C_\epsilon+\epsilon\|u\|_{H^2(\Omega)}^{2}$ for any fixed small $\epsilon>0.$

\end{enumerate}

\end{lem}

\begin{proof} Note that $$\nabla F(u) =   -(\kappa+i\beta)\left\{{\frac{(p+1)}{2}}|u|^{p-1}\nabla u+{\frac{(p-1)}{2}}|u|^{p-3}u^2\nabla\bar{u}\,\right\}+\gamma\,\nabla u.$$  Therefore,
$$|\nabla F(u)|\le C(|u|^{p-1}|\nabla u|+|\nabla u|).$$ The case  $N=1$ follows from the Sobolev embedding $H^1(\Omega)\hookrightarrow L^\infty(\Omega)$ and Lemma \ref{H1global}.

Let $N=2$ and $1<p<\infty$.  Then, by using the Gagliardo-Nirenberg inequality \begin{equation}\label{GagN2}\|\nabla u\|_{L^s(\Omega)}\le\|u\|_{H^2(\Omega)}^{\theta}\|u\|_V^{1-\theta},\end{equation}  where $\displaystyle s>2$ and $0<\displaystyle\theta=1-\frac{2}{s}<1,$ the Sobolev embedding $H^1(\Omega)\hookrightarrow L^q(\Omega) $ ($q\ge 1)$, and Lemma \ref{H1global}, we have \begin{equation}\label{713}\|F(u)\|_V\le C\left(\|u\|_{q}^{(p-1)}\|\nabla u\|_s+\|\nabla u\|_{L^2(\Omega)}\right)\\
\le C\left(\|u\|_{V}^{(p-1)}\|u\|_V^{1-\theta}\|u\|_{H^2(\Omega)}^{\theta}+\|u\|_V\right)\end{equation} with $q>2(p-1)$ and $q\ge 1$.

 Let $N\ge 3$ and $1<p<\frac{N}{N-2}$; then we have the following Gagliardo-Nirenberg inequality:
\begin{equation}
  \|\nabla u\|_{L^s(\Omega)}\le \|u\|_{H^2(\Omega)}^{\theta}\|u\|_V^{1-\theta}
\end{equation} where $\displaystyle \frac{2N}{N-2}>s=\frac{4N}{2N-2(N-2)(p-1)}>2$ and $0<\displaystyle\theta=\frac{(s-2)N}{2s}<1.$

For $q=\frac{2N}{N-2}$, \eqref{713} can be rewritten as

\begin{multline}
\|F(u)\|_V\le C\left(\||u|^{p-1}|\nabla u|\|_{L^2(\Omega)}+|\nabla u|\|_{L^2(\Omega)}\right)\le C\left(\|u\|_{q}^{(p-1)}\|\nabla u\|_s+\|u\|_V\right)\\
\le C\left(\|u\|_{V}^{(p-1)}\|u\|_{H^2(\Omega)}^{\theta}\|u\|_V^{1-\theta}+\|u\|_V\right).
\end{multline}

Let $N=3$ and $p\ge 3$, then we have the Gagliardo-Nirenberg inequality
$$\|u\|_{\infty}\le C\|u\|_{H^2(\Omega)}^{\frac{3}{4}}\|u\|_{2}^{\frac{1}{4}}.$$  Therefore,

\begin{equation}\|F(u)\|_V\le C\left(\|u\|_{\infty}^{p-1}\|\nabla u\|_{L^2(\Omega)}+\|\nabla u\|_{L^2(\Omega)}\right)\\
\le C\left(\|u\|_{H^2(\Omega)}^{\frac{3}{4}(p-1)}\|u\|_{L^2(\Omega)}^{\frac{p-1}{4}}\|\nabla u\|_{2}+\|u\|_V\right).
\end{equation}
Note that if in addition, $p<\frac{11}{3}$, then $\frac{3}{4}(p-1)<2.$
\end{proof}

Now,   we will prove the following estimate:
	 	 	   		{\small  \begin{align}
	 	 	   			\|u_t(t) \|_V + \|u(t)\|_{H^2(\Om)}  \leq C({u_0},T_{\max}) &+ C\left(\int_0^t  \|F_t(u(s),u_t(s)) \|_V ds \right.\nonumber\\
	 	 	   			&+\left.  \Big[\int_0^t\!\! \int_{\Gamma_1} \!\!|F_t(u(s),u_t(s))|^2 d\Gamma\,ds\Big]^{1/2} \right)\label{es}
	 	 	   			\end{align}} where the constants  $ C({u_0},T_{\max})$ and $ C $ are given more precisely later.\\
	 	 	   		
	 	 	   		Substituting $ f=F(u) $ in \eqref{LSf},  using $ g=-\left.f\right|_{\Ga_1} $ and \eqref{comp_initial}, \eqref{yt} can be written as
	 	 	   		{\small \begin{equation}
	 	 	   			\label{yt1}u_t=(\lambda+i\alpha)e^{A\,t}\triangle\,u_0+ e^{A\,t}F(u_0)+\int_0^t\! e^{A\,(t-s)}{F}_t(u(s), u_t(s))ds-A\!\int_0^t\! e^{A\,(t-s)}\mathcal{N}F_t(u(s), u_t(s))ds.
	 	 	   			\end{equation}}
	 	 	   		On the other hand, employing the Lemma \ref{gagliardo}, it follows that
	 	 	   		{ \begin{equation}
	 	 	   			\label{F0}\begin{aligned}
	 	 	   			\|e^{A\,t}F(u_0)\|_V&\leq C(\|u_0\|_{H^2(\Omega)}).
	 	 	   			\end{aligned}
	 	 	   			\end{equation}}
	 	 	   		
	 	 	   		Then, combining (\ref{yt1}), \eqref{F0}, and the continuity of the  map  $ \mathcal{L} $ given in  \eqref{Lmap} (whose bound, say is $ \eta $), we obtain
	 	 	   		{\small    \begin{equation}
	 	 	   			\label{yt2}\begin{aligned}
	 	 	   			\|u_t(t)\|_{V}&\leq \underbrace{C\left(\|\triangle\,u_0\|_V,\|u_0\|_{H^2(\Om)}\right)}_{:=C_{u_0}}+\int_0^t\!\! \|{F}_t(u(s), u_t(s))\|_Vds+ \|\mathcal{L}(F_t(u(s), u_t(s)))\|_{C([0,T];V)}\\
	 	 	   			&\leq C_{u_0}+\int_0^t\!\! \|{F}_t(u(s), u_t(s))\|_V\,ds+ \eta\|F_t(u(s), u_t(s))\|_{L^2(0,t; L^2(\Ga_1))}\\
	 	 	   			&=C_{u_0}+\int_0^t\!\! \|{F}_t(u(s), u_t(s))\|_V\,ds+ \eta\left[\int_0^t\|F_t(u(s), u_t(s))\|_{L^2(\Ga_1)}^2\,ds\right]^{1/2}\\
	 	 	   			&\leq C_{u_0}+ \,\int_0^t\!\! \|{F}_t(u(s), u_t(s))\|_V\,ds
	 	 	   			+ \eta\left[\int_0^t\|F_t(u(s), u_t(s))\|_{L^2(\Ga_1)}^2\,ds\right]^{1/2}.
	 	 	   			\end{aligned}
	 	 	   			\end{equation}}
	 	 	   		Now, from \eqref{remH2space} and the continuity of the trace map $ \ga: H^1(\Om) \rightarrow H^{1/2}(\Ga_1) $ (whose  embedding constant, say $ \varsigma $), we have
	 	 	   		\begin{equation}
	 	 	   		\begin{aligned}
	 	 	   		\label{H2est}\|u(t)\|_{H^2(\Om)}&\leq C_1\left(\|u_t(t)\|_{V}+\|F(u(t))\|_{V}+\|F(u(t))\|_{H^{1/2}(\Ga_1)}\right)\\
	 	 	   		&\leq C_1\left(\|u_t(t)\|_{V}+(1+\varsigma)\|F(u(t))\|_{V}\right).
	 	 	   		\end{aligned}
	 	 	   		\end{equation}
Thus, combining Lemma \ref{gagliardo}, \eqref{yt2} and \eqref{H2est}, \eqref{es} follows.

	 	 	   		Now, we employ an inequality derived from the Gagliardo - Nirenberg inequality (see \cite{Bechouche}):
	 	 	   		\begin{align}\label{H2}
	 	 	   		\|z(t)\|_{H^2(\Om)} &\leq {C}\,\left(\|z\|_{L^2(\Om)}+\|\triangle\,z\|_{L^2(\Om)}\right)\\
	 	 	   		&\leq\nonumber \widetilde{C}\,\left(\|z\|_{V}+\|\triangle\,z\|_{L^2(\Om)}\right),\,\forall\,z\in\,H^2(\Om)\,.
	 	 	   		\end{align}
	 	 	   		On the other hand, from Lemma \ref{l:est} and \eqref{H2}, we obtain:
	 	 	   		\begin{equation}\begin{aligned}
	 	 	   		\label{integral1}\int_0^t\! \|{F}_t(u(s), u_t(s))\|_V\,ds &\leq C\,\int_0^t\!\|w(s)\|_V\left(\|u(s)\|_{H^2(\Om)}^2+1\right)\,ds.
	 	 	   		\end{aligned}
	 	 	   		\end{equation}

	 	 	   		Having in mind that $ \widetilde{\gamma} $  denotes a trace operator -restriction to the boundary $ \Ga_1, $ a mixed trace-interpolation inequality results
	 	 	   		\begin{equation}
	 	 	   		\label{trace}
	 	 	   		\|\widetilde{\gamma} u\|^2_{L^2(\Gamma_1)} \leq C \|u\|_{H^1(\Omega) } \|u\|_{L^2(\Omega) } \,.
	 	 	   		\end{equation}
	 	 	   		This implies
	 	 	   		\begin{equation}\label{Ft1}
	 	 	   		\|\widetilde{\gamma} F_t(u,w) \|^2_{L^2(\Gamma_1)} \leq C \|F_t(u,w)\|_{H^1(\Omega) } \|F_t(u,w)\|_{L^2(\Omega) }\,.
	 	 	   		\end{equation}

	 	 	   		Recall that for $N=1$, we have $H^1(\Omega)\hookrightarrow L^{\infty}(\Omega)$; for $N=2$, $H^1(\Omega)\hookrightarrow L^{q}(\Omega)$ if $1\le q<\infty$; for $N\ge 3$, $H^1(\Omega)\hookrightarrow L^{q}(\Omega)$ if $1\le q\le\frac{2N}{N-2}$.  Using these Sobolev embeddings and Lemma \ref{H1global}, for $p\le \frac{N}{N-2}$ if $N\ge 3$ we get:

	 	 	   	 	\begin{equation}
	 	 	   		\label{est_L2}\|F_t(u,w)\|_{L^2(\Om)}\leq C\left(\|\,|u|^{p-1}\,w\|_{L^2(\Om)}+\|w\|_{L^2(\Omega)}\right)\leq CE(0)\|w\|_V.
	 	 	   		\end{equation}
\begin{lem} Let $\alpha,\beta,\lambda,\kappa>0$.  Then, there exists $M>0$ such that
  $$\sup_{t\in [0,T_{max})}\|w(t)\|_{V}<M<\infty.$$
\end{lem}	 	 	   	
\begin{proof}
  Multiply, \eqref{weq} by $-\alpha{\Delta \bar{w}}$, integrate over $\Omega$ and take the real parts.  Then, by also using \eqref{est_L2} we have
    \begin{multline}\label{wGlobal}\frac{\alpha}{2}\frac{d}{dt}\|\nabla w\|_{L^2(\Omega)}^2+\alpha\|w_t\|_{L^2(\Gamma_1)}^2+\alpha\lambda\|\Delta w\|_{2}^2= -\alpha\text{Re}(F_t(u,w),\Delta w)_{L^2(\Omega)}\\
\le\frac{1}{4\epsilon}\|F_t(u,w)\|_{L^2(\Omega)}^2+\epsilon\|\Delta w\|_{L^2(\Omega)}^2\le C(E(0))\|w\|_V^2+\epsilon\|\Delta w\|_{L^2(\Omega)}^2.\end{multline}  The above inequality gives
$$\|w(t)\|_V^2\le \|w_0\|_V^2+C(E(0))\int_0^t\| w(s)\|_V^2ds.$$  Now, from Gronwall's inequality the desired property follows.
\end{proof}	

\begin{rem}
  From \eqref{wGlobal}, we note that we are able to control the term $ \epsilon\|\Delta w\|_{L^2(\Omega)}^2$ with the term $\alpha\lambda\|\Delta w\|_{2}^2$ at the left hand side.  This is an important property for CGLE.  For instance, in the case of the nonlinear Schrödinger equation \cite{Dynamic_BC}, one does not have such a control since $\lambda=0.$
\end{rem}

Now, it follows from \eqref{Ft1} and \eqref{est_L2} that

  		\begin{equation}\label{FtG1}
	 	 	   		\|\widetilde{\gamma} F_t(u,w) \|^2_{L^2(\Gamma_1)} \leq C \|F_t(u,w)\|_{H^1(\Omega) }\,.
	 	 \end{equation}
Moreover, from \eqref{integral1}, we get

\begin{equation}\begin{aligned}\label{FtV}
	 	 	   	\int_0^t\! \|{F}_t(u(s), u_t(s))\|_V\,ds &\leq C+C\,\int_0^t\!\|u(s)\|_{H^2(\Om)}^2\,ds.
	 	 	   		\end{aligned}
	 	 	   		\end{equation} 	   	
But the right hand side of \eqref{FtV} is bounded by a constant due to Lemma \eqref{H1global} (since $\lambda>0$).
Combining this with \eqref{es}, we obtain

$$\|(u,u_t)\|_{X_T}\le C.$$  Hence, we obtained a uniform bound for $\|(u,u_t)\|_{X_T}$ on $[0,T_{\max}).$ The proof of the global existence of strong solutions is complete.

\section{Weak Solutions}
In this section, we will prove the global existence and uniqueness of weak solutions given in Theorem \ref{main_density_theo}.

So, let $$u_0\in  V \text{ such that }\gamma_0\varphi\in L^{p+1}(\Gamma_1)$$ and $$ \{u_{\mu,0}\} $$ be smooth enough that
	 \begin{equation}
	 \label{conv_initial}
	 u_{\mu,0} \longrightarrow u_0 \quad \text{ in } Q\,
	 \end{equation} and
	 for each $ \mu\in\,\mathbb{N}, $ $ u_\mu $  is the unique strong solution of \eqref{GL_equation_1} with initial data $ \{u_{\mu,0}\}.$  Then, $u_\mu$ solves
	 \begin{equation}\label{GL_equation_density_mu}
	 \begin{cases}
	 \partial_tu_{\mu}-(\lambda+i\alpha)\triangle u_\mu+(\kappa+i\,\beta)|u_\mu|^{p-1}\,u_\mu-\gamma u_\mu=0 &\text { in } \Omega\times \mathbb{R_+},\\
	 \frac{\partial u_\mu}{\partial \nu} = -\partial_tu_{\mu} &\text { on } \Gamma_1\times \mathbb{R_+},\\
	 u_\mu=0  &\text { on } \Gamma_0\times \mathbb{R_+},\\
	 u_\mu(0)=u_{\mu,0}  &\text { in } \Omega.
	 \end{cases}
	 \end{equation}

	Now, we define $ z_{\mu,\sigma}:=u_\mu-u_\sigma$ for  $\mu, \sigma\in\,\mathbb{N}. $ We prove the following for $ \{z_{\mu,\sigma}\}$:
	
	 \begin{lem}\label{H1global_density}	
${E}_{\mu,\sigma}(t)\le {E}_{\mu,\sigma}(0)\exp(CT),\,t\in [0,T]$, where the definition of ${E}_{\mu,\sigma}$ is given in \eqref{energy_density}.
	 \end{lem}
	
	 \begin{proof}
Setting $G(z_{\mu, \sigma}):=|u_\mu|^{p-1}\,u_\mu-|u_\sigma|^{p-1}\,u_\sigma $, we observe that	 	
	 	\begin{multline}
	 	\Re(-\alpha\Delta z_{\mu, \sigma}+\beta\, G(z_{\mu, \sigma}),z_{\mu, \sigma}^\prime)\\
	 	 = Re\left(-\alpha\Delta z_{\mu, \sigma}+\beta\, G(z_{\mu, \sigma}),(\lambda+i\alpha)\triangle z_{\mu, \sigma}-(\kappa+i\beta)\, G(z_{\mu, \sigma})+\gamma z_{\mu, \sigma}\right)\\
	 	=-\alpha\lambda\|\Delta z_{\mu, \sigma}\|_{L^2(\Omega)}^2+\alpha\gamma \|\nabla z_{\mu, \sigma}\|_{L^2(\Omega)}^2-\alpha\gamma Re\left(\frac{\partial z_{\mu, \sigma}}{\partial n},z_{\mu, \sigma}\right)_{L^2(\Gamma_1)}-\kappa\beta\|G(z_{\mu, \sigma})\|_{2}^{2}+\beta\gamma(G(z_{\mu, \sigma}),z_{\mu, \sigma})\\
	 	+\alpha Re(\Delta z_{\mu, \sigma},(\kappa+i\beta)G(z_{\mu, \sigma}))+\beta Re(G(z_{\mu, \sigma}),(\lambda+i\alpha)\triangle z_{\mu, \sigma})\\
	 	=-\alpha\lambda\|\Delta z_{\mu, \sigma}\|_{L^2(\Omega)}^2+\alpha\gamma \|\nabla z_{\mu, \sigma}\|_{L^2(\Omega)}^2-\alpha\gamma Re\left(\frac{\partial z_{\mu, \sigma}}{\partial n},z_{\mu, \sigma}\right)_{L^2(\Gamma_1)}-\kappa\beta\|f(u)\|_{2}^{2}+\beta\gamma(G(z_{\mu, \sigma}),z_{\mu, \sigma})\\
	 	-\alpha Re(\nabla z_{\mu, \sigma},(\kappa+i\beta)\nabla G(z_{\mu, \sigma}))+\alpha \Re\left(\frac{\partial z_{\mu, \sigma}}{\partial n},(\kappa+i\beta)G(z_{\mu, \sigma})\right)_{L^2(\Gamma_1)}\\
	 	-\beta \Re\left(\nabla G(z_{\mu, \sigma})),(\lambda+i\alpha)\nabla z_{\mu, \sigma}\right)+\beta \Re\left(G(z_{\mu, \sigma}),(\lambda+i\alpha)\frac{\partial z_{\mu, \sigma}}{\partial n}\right)_{L^2(\Gamma_1)}.
	 	\end{multline}
On the other hand, it follows that
	 	\begin{multline}
	 	\left|-\alpha\gamma Re\left(\frac{\partial z_{\mu, \sigma}}{\partial n},z_{\mu, \sigma}\right)_{L^2(\Gamma_1)}\right| \le \frac{\alpha\gamma}{2}\left|z_{\mu, \sigma}^\prime,z_{\mu, \sigma})_{L^2(\Gamma_1)}\right|\\
	 	\leq  {\alpha\gamma}M\,\|z_{\mu, \sigma}^\prime\|_{L^2(\Gamma_1)}\,\|z_{\mu, \sigma}\|_{L^2(\Gamma_1)}\\
\le	 	\frac{\alpha\gamma}{2}\left(C_{\epsilon_1}\|z_{\mu, \sigma}\|_{L^2(\Gamma_1)}^2+\epsilon_1\left\|z_{\mu, \sigma}^\prime\right\|_{L^{2}(\Gamma_1)}^{2}\right)\\
	 	\le \frac{\alpha\gamma}{2}\left(C_{\epsilon_1}\|\nabla z_{\mu, \sigma}\|_{L^2(\Omega)}^2+\epsilon_1\left\|z_{\mu, \sigma}^\prime\right\|_{L^{2}(\Gamma_1)}^{2}\right).
	 	\end{multline}
	 	
	 	Alternatively, by using $\displaystyle \frac{\partial z_{\mu, \sigma}}{\partial \nu}|_{\Gamma_1}=-z_{\mu, \sigma}^ \prime$ one also has
	 	
	 	\begin{equation}
	 	-\alpha\gamma \Re\left(\frac{\partial z_{\mu, \sigma}}{\partial n},z_{\mu, \sigma}\right)_{L^2(\Gamma_1)}= \frac{\alpha\gamma}{2}\frac{d}{dt}\|z_{\mu, \sigma}\|_{L^2(\Gamma_1)}^2.
	 	\end{equation}
	 	

	 	\begin{multline}\label{derivative}
	 	\Re(-\alpha\Delta z_{\mu, \sigma}+\beta\, G(z_{\mu, \sigma}),z_{\mu, \sigma}^\prime)=\frac{\alpha}{2} \frac{d}{dt}\|\nabla z_{\mu, \sigma}\|_{L^2(\Omega)}^2-\Re\left(\frac{\partial z_{\mu, \sigma}}{\partial n},z_{\mu, \sigma}^\prime\right)_{L^2(\Gamma_1)}+\Re(\beta\, G(z_{\mu, \sigma}),z_{\mu, \sigma}^\prime) \\
	 	=\frac{\alpha}{2} \frac{d}{dt}\|\nabla z_{\mu, \sigma}\|_{L^2(\Omega)}^2+\alpha\|z_{\mu, \sigma}^\prime\|_{L^2(\Gamma_1)}^2+\Re(\beta\, G(z_{\mu, \sigma}),z_{\mu, \sigma}^\prime).
	 	\end{multline}
Now,  taking account \eqref{H2}, from the following local Lipschitz estimates:
\begin{align}
\label{GL2}\|G(z_{\mu, \sigma})\|_{L^2(\Omega)} &\leq \left(\|u_{\mu}\|_{H^2},\|u_{\sigma}\|_{H^2}\right)\,\|z_{\mu, \sigma}\|_{L^2(\Omega)} \leq C\left(\|u_{\mu,0}\|_{H^2},\|u_{\sigma,0}\|_{H^2}\right)\|z_{\mu, \sigma}\|_{L^2(\Omega)}\\
\label{GV}\|G(z_{\mu, \sigma})\|_V &\leq \left(\|u_{\mu}\|_{H^2},\|u_{\sigma}\|_{H^2}\right)\,\|z_{\mu, \sigma}\|_V\leq C\left(\|u_{\mu,0}\|_{H^2},\|u_{\sigma,0}\|_{H^2}\right)\|z_{\mu, \sigma}\|_V,
\end{align}
we have
	 	\begin{multline}\label{derivative1}
	 	\Re(-\alpha\Delta z_{\mu, \sigma}+\beta\, G(z_{\mu, \sigma}),z_{\mu, \sigma}^\prime)
	 	=\frac{\alpha}{2} \frac{d}{dt}\|\nabla z_{\mu, \sigma}\|_{L^2(\Omega)}^2+\alpha\|z_{\mu, \sigma}^\prime\|_{L^2(\Gamma_1)}^2+\Re(\beta\, G(z_{\mu, \sigma}),z_{\mu, \sigma}^\prime)\\
	 	\leq \frac{\alpha}{2} \frac{d}{dt}\|\nabla z_{\mu, \sigma}\|_{L^2(\Omega)}^2+\alpha\|z_{\mu, \sigma}^\prime\|_{L^2(\Gamma_1)}^2 +\beta\,C^\prime\,\|\nabla\,G(z_{\mu, \sigma})\|_{L^2(\Omega)}\,\|z_{\mu, \sigma}^\prime\|_{L^2(\Omega)}\\
	 		 	\leq \frac{\alpha}{2} \frac{d}{dt}\|\nabla z_{\mu, \sigma}\|_{L^2(\Omega)}^2+\alpha\|z_{\mu, \sigma}^\prime\|_{L^2(\Gamma_1)}^2 +\beta\,C_{\epsilon_2}\|\nabla\, z_{\mu, \sigma}\|_{L^2(\Omega)}^2 + \beta\,C\,\epsilon_2\left[\|\triangle\,z_{\mu, \sigma}\|_{L^2(\Om)}^2+\|G(z_{\mu, \sigma})\|^{2}_{L^{2}(\Om)}+
	 		 	\|z_{\mu, \sigma}\|_{V}^2\right]
	 	\end{multline}
	 	\begin{multline}
	 	-\alpha Re(\nabla z_{\mu, \sigma},(\kappa+i\beta)\nabla G(z_{\mu, \sigma}))-\beta Re(\nabla G(z_{\mu, \sigma}),(\lambda+i\alpha)\nabla z_{\mu, \sigma})
	 	=-(\alpha\kappa+\beta\lambda)Re(\nabla(G(z_{\mu, \sigma})),\nabla z_{\mu, \sigma})
	 	\\\le  |\alpha\kappa+\beta\lambda| Re(\nabla G(z_{\mu, \sigma}),\nabla z_{\mu, \sigma})\\
	 	\le |\alpha\kappa+\beta\lambda|\,\|G(z_{\mu, \sigma})\|_V\,\|z_{\mu, \sigma}\|_V
	 	\le C_1\,\|z_{\mu, \sigma}\|_V^2
	 	\end{multline}
	 	and \begin{equation}
	 	|	\beta\gamma(G(z_{\mu, \sigma}),z_{\mu, \sigma})| \leq \beta\,\ga\,C_2\,\|z_{\mu, \sigma}\|_V^2\,.
	 	\end{equation}

	 	\begin{multline}
	 	\alpha \Re\left(\frac{\partial z_{\mu, \sigma}}{\partial n},(\kappa+i\beta)G(z_{\mu, \sigma})\right)_{L^2(\Gamma_1)} +\beta \Re\left(G(z_{\mu, \sigma}),(\lambda+i\alpha)\frac{\partial z_{\mu, \sigma}}{\partial n}\right)_{L^2(\Gamma_1)}\\
	 	=(\alpha\kappa+\beta\lambda)\Re\left(G(z_{\mu, \sigma}),\frac{\partial z_{\mu, \sigma}}{\partial n}\right)_{L^2(\Gamma_1)}
\\
\le	 	|\alpha\kappa+\beta\lambda|\left(C_{\epsilon_3}\|G(z_{\mu, \sigma})\|_{L^2(\Gamma_1)}^2+\epsilon_3\left\|z_{\mu, \sigma}^\prime\right\|_{L^{2}(\Gamma_1)}^{2}\right)\\
\le |\alpha\kappa+\beta\lambda|\left(C_{\epsilon_3}\|\nabla z_{\mu, \sigma}\|_{L^2(\Omega)}^2+\epsilon_3\left\|z_{\mu, \sigma}^\prime\right\|_{L^{2}(\Gamma_1)}^{2}\right).
	 	\end{multline}
	 	
Collecting all the estimates given above, we get
\begin{eqnarray}
\begin{aligned}
\label{inequality_energy_density}\frac{\alpha}{2} &\frac{d}{dt}\|\nabla z_{\mu, \sigma}\|_{L^2(\Omega)}^2+\alpha\left(1-\frac{\ga\,\epsilon_1}{2}-|\alpha\kappa+\beta\lambda|\epsilon_3\right)\|z_{\mu, \sigma}^\prime\|_{L^2(\Gamma_1)}^2\\&+\left(\alpha\lambda-\beta\,C\,\epsilon_2\right)\|\Delta z_{\mu, \sigma}\|_{L^2(\Omega)}^2+(\kappa\beta-\beta\,C\,\epsilon_2)\|G(z_{\mu, \sigma})\|_{2}^{2}\\
&\leq C\,\|\nabla z_{\mu, \sigma}\|_{L^2(\Omega)}^2\,.
\end{aligned}
\end{eqnarray}
Considering $ \epsilon_i, i=1,2,3$ small enough and integrating \eqref{inequality_energy_density} in $ t\in\,[0,T], $ we have
	 	
	 	$${E}_{\mu,\sigma}(t)\le C{E}_{\mu,\sigma}(0)+C\int_0^t{E}_{\mu,\sigma}(s)ds,$$
	 	where
	 	
	 	\begin{equation}
	 	\label{energy_density}{E}_{\mu,\sigma}(t):=\|\nabla z_{\mu, \sigma}(t)\|_{L^2(\Omega)}^2+C_1\int_0^t\|z_{\mu, \sigma}^\prime(s)\|_{L^2(\Gamma_1)}^2\,ds+C_2\int_0^t\|\Delta z_{\mu, \sigma}(s)\|_{L^2(\Omega)}^2\,ds
	 	\end{equation}
	 	Now, using Gronwall's inequality, we conclude that \begin{equation}\label{EE02}{E}_{\mu,\sigma}(t)\le {E}_{\mu,\sigma}(0)\exp(CT)\end{equation} for $t\in [0,T]$.
	 \end{proof}
	
	 From \eqref{conv_initial} and Lemma \ref{H1global_density},  we conclude that there exists a function $ u $ such that, for all $ T>0, $ we have
	
	 \begin{eqnarray}
	 \label{Vstrong}u_\mu&\longrightarrow& u \quad \,\,\,\,\,\text{ in } \quad C([0,T];\,V)\hookrightarrow\,L^2(0,T;\,\ld),\\
	 	 \label{u'strong}u_\mu^\prime&\longrightarrow& u^\prime \quad\,\,\,\, \text{ in } \quad L^2(0,T;\,L^2(\Ga_1)),\\
	 	 	 \label{Deltastrong}\Delta\,u_\mu&\longrightarrow& \Delta\,u \quad \text{ in } \quad L^2(0,T;\,\ld).
	 \end{eqnarray}
	
	 Moreover, from \eqref{H2}, \eqref{Vstrong} and \eqref{Deltastrong}, we have
\begin{equation}
	 	 	 	 \label{H2strong}u_\mu\longrightarrow u \quad \text{ in } \quad L^2(0,T;\,H^2(\Om)).
\end{equation}

Considering the last convergence and those given by \eqref{Vstrong}-\eqref{Deltastrong}, it follows that $ u $ is a weak solution in the sense of Definition \ref{Defweak}.

Now, let $ u_1 $ and $ u_2 $ be two solutions of \eqref{GL_equation}. Then $ w=u_1-u_2 $ satisfies
\begin{equation}\label{GL_equation_w}
\begin{cases}
w_t-(\lambda+i\alpha)\triangle w+(\kappa+i\,\beta)\left(|u_1|^{p-1}\,u_1-|u_2|^{p-1}\,u_2\right)-\gamma w=0 &\text { in } \Omega\times \mathbb{R_+},\\
\frac{\partial w}{\partial \nu} = -w_t &\text { on } \Gamma_1\times \mathbb{R_+},\\
w=0  &\text { on } \Gamma_0\times \mathbb{R_+},\\
w(0)=0  &\text { in } \Omega.
\end{cases}
\end{equation}

Since  $ w\in\,H^2(\Om) $ and $ w_t\in\,L^2(\Om), $ a. e. $ t\in\,[0,T], $   we can repeat the procedure used in Lemma \ref{H1global_density} with which we can get $w=0  $ for a.e. $ (x,t)\in\,\Omega\times [0,T]. $

\section{Inviscid Limits}  In this section, we will prove Theorems \ref{Inviscid0} and \ref{Inviscid}. which state that the solutions of the CGLE subject to dynamic boundary conditions converge to the solution of NLS subject to same type of boundary conditions as the parameter pair $\epsilon=(\lambda,\kappa)$ tend to zero.

We know from Lemma \ref{H1global} that $u_\epsilon$ is uniformly bounded in $L^{\infty}(0,T;V)$, $\partial_tu_{\epsilon}$ is bounded in $L^2(0,T;L^2(\Gamma_1))$, and $|u_\epsilon|^{p-1}u_{\epsilon}$ is uniformly bounded in $L^\infty(0,T;L^{(p+1)'}(\Omega))$.  Using the main equation and Lemma \ref{H1global} again, we can also see that $\partial_t u_\epsilon$ is uniformly bounded in $L^2(0,T; V')$.  From these bounds, it follows that there exists a subsequence of $u_{\epsilon}$ (denoted same) such that

\begin{eqnarray}
  u_\epsilon \rightharpoonup u &  & \text{ weakly-star in }   L^{\infty}(0,T;V);\\
  \partial_tu_{\epsilon} \rightharpoonup u_t & &  \text{ weakly in }   L^{2}(0,T;L^2(\Gamma_1));\\
  \partial_tu_\epsilon\rightharpoonup u_t &  & \text{ weakly in }   L^{2}(0,T;V');\\
  |u_\epsilon|^{p-1}u_{\epsilon}\rightharpoonup \xi &  & \text{ weakly-star in }   L^\infty(0,T;L^{(p+1)'}(\Omega)).
\end{eqnarray}

Recall that the bounded sets in $X\equiv \{u\in  L^{2}(0,T;V)|u_t\in  L^{2}(0,T;V')\}$ are relatively compact in $L^2(0,T;L^2(\Omega))$.  Therefore, we can assume $u_{\epsilon}\rightarrow u$ strongly in  $L^2(0,T;L^2(\Omega))$ and $u_{\epsilon}\rightarrow u$ a.e. in $(0,T)\times \Omega.$  This implies that $|u_{\epsilon}|^{p-1}u_\epsilon\rightarrow |u|^{p-1}u$ a.e. in $(0,T)\times \Omega.$ But we have $|u_{\epsilon}|^{p-1}u_\epsilon$ is bounded in $L^\infty(0,T;L^{(p+1)'}(\Omega))$. Therefore $\xi\equiv |u|^{p-1}u.$   Now, letting $\epsilon\rightarrow 0$, we see that $u$ solves the idbvp for the NLS.   This completes the proof of Theorem \ref{Inviscid0}.

In order to prove Theorem \ref{Inviscid}, we will restrict our attention to only dimension $N=2$, $p=3$, and $\beta>0$, for which we know that there is a global strong solution for the NLS \cite{Dynamic_BC}.  To this end, let $u_\epsilon$ be the global strong solution of \eqref{GL_equation_1} with initial datum $u_\epsilon(0)=u^0_\epsilon$ and $p=3$.  Suppose $u$ is a global strong solution of the cubic defocusing nonlinear Schrödinger equation with dynamic boundary conditions below ($\alpha,\beta>0$):
\begin{equation}\label{NLSeq}
\begin{cases}
u_t-i\alpha\triangle u+i\beta|u|^{2}u-\gamma u=0 &\text { in } \Omega\times (0,T),\\
\displaystyle\frac{\partial u}{\partial \nu} = -u_t &\text { on } \Gamma_1\times (0,T),\\
u=0  &\text { on } \Gamma_0\times (0,T),\\
u(0)=u_0  &\text { in } \Omega.
\end{cases}
\end{equation}

Moreover, let us suppose that $u_\epsilon^0\rightarrow u_0$ as $\epsilon\rightarrow 0$ in $V$.

Set $w= u_\epsilon - u.$  Then, $w$ solves the following problem:
\begin{equation}\label{wDiffeq}
\begin{cases}
w_t-i\alpha\triangle w+i\beta\left(|u_\epsilon|^{2}u_\epsilon-|u|^{2}u\right)-\gamma w-\lambda\Delta u_\epsilon+\kappa|u_\epsilon|^{2}u_\epsilon=0 &\text { in } \Omega\times (0,T),\\
\displaystyle\frac{\partial w}{\partial \nu} = -w_t &\text { on } \Gamma_1\times (0,T),\\
w=0  &\text { on } \Gamma_0\times (0,T),\\
w(0)=w_0^\epsilon\equiv u_\epsilon^0-u_0  &\text { in } \Omega.
\end{cases}
\end{equation}

We multiply \eqref{wDiffeq} by $-\Delta \bar{w}$, integrate over $\Omega\times (0,t)$, and take the real parts:

\begin{multline}\label{impmain}
\frac{1}{2}\|\nabla w(t)\|_{L^2(\Omega)}^2+\int_0^t\|w_t\|_{L^2(\Gamma_1)}^2ds+\beta\text{Im}\int_0^t(|u_\epsilon|^{2}u_\epsilon-|u|^{2}u,\Delta w)_{L^2(\Omega)}ds\\-\gamma\int_0^t\|\nabla w\|_{L^2(\Omega)}^2ds+\ga\int_0^t\text{Re}(w,\partial_\nu w)_{L^2(\Gamma_1)}ds
+\lambda\text{Re}\int_0^t(\Delta u_\epsilon,\Delta w)_{L^2(\Omega)}ds\\
-\kappa\int_0^t\text{Re}(|u_\epsilon|^{2}u_\epsilon,\Delta w)_{L^2(\Omega)}ds=\frac{1}{2}\|\nabla w_0^\epsilon\|_{L^2(\Omega)}^2.
\end{multline}

Using the boundary condition, we estimate

\begin{equation}\frac{\gamma}{2}\text{Re}(w,\partial_\nu w)_{L^2(\Gamma_1)}\le \frac{1}{2}\|w_t\|_{L^2(\Gamma_1)}^2+\frac{\gamma^2}{8}\|w\|_{L^2(\Gamma_1)}^2\le \frac{1}{2}\|w_t\|_{L^2(\Gamma_1)}^2+\frac{C\gamma^2}{8}\|w\|_{V}^2,\end{equation} where the last inequality follows from the trace theorem.
On the other hand, we calculate

\begin{multline}\label{imp0}
  \beta\text{Re}\int_0^t(|u_\epsilon|^{2}u_\epsilon-|u|^{2}u,\Delta w)_{L^2(\Omega)}ds \\
  =  -\beta\text{Re}\int_0^t \left(\nabla\left(|u_\epsilon|^{2}u_\epsilon-|u|^{2}u\right),\nabla w\right)_{L^2(\Omega)}ds + \beta\text{Re}\int_0^t(|u_\epsilon|^{2}u_\epsilon-|u|^{2}u,\frac{\partial w}{\partial \nu})_{L^2(\Gamma_1)}ds\\
  \le \beta\int_0^t \|\nabla\left(|u_\epsilon|^{2}u_\epsilon-|u|^{2}u\right)\|_{L^2(\Omega)}\|\nabla w\|_{L^2(\Omega)}ds+\frac{1}{2}\int_0^t\|w_t\|_{L^2(\Gamma_1)}^2ds\\
  +\frac{\beta^2}{2}\int_0^t\||u_\epsilon|^{2}u_\epsilon-|u|^{2}u\|_{L^2(\Gamma_1)}^2ds.
\end{multline}

One can prove that (by slightly modifying the calculations in \cite[Lemma 3.3]{Ozsar2015}):
\begin{multline}\label{imp1}\|\nabla\left(|u_\epsilon|^{2}u_\epsilon-|u|^{2}u\right)\|_{L^2(\Omega)} \le C\|u_\epsilon-u\|_{L^2(\Omega)}\left(\|u_\epsilon\|_\infty+\|u\|_\infty\right)^2\\
+C\|u_\epsilon-u\|_{4}\left(\|u_\epsilon\|_{\infty}+\|u\|_{\infty}\right)\left(\|\nabla u_\epsilon\|_{4}+\|\nabla u\|_{4}\right)\\
\le C\|u_\epsilon-u\|_V\left(\|u_\epsilon\|_{H^2(\Omega)}+\|u\|_{H^2(\Omega)}\right)^2\\ +C\|u_\epsilon-u\|_{V}\left(\|u_\epsilon\|_{H^2(\Omega)}+\|u\|_{H^2(\Omega)}\right)\left(\|u_\epsilon\|_{H^2(\Omega)}^{\frac{1}{2}}\|u_\epsilon\|_{V}^{\frac{1}{2}}+\|u\|_{H^2(\Omega)}^{\frac{1}{2}}\|u\|_{V}^{\frac{1}{2}}\right)\\
\le C\|w\|_{V},\end{multline} where the second inequality above follows from the facts $V\hookrightarrow L^2(\Omega)$, $H^2(\Omega)\hookrightarrow L^\infty(\Omega)$ $(N=2)$, $H^1(\Omega)\hookrightarrow L^4(\Omega)$ ($N=2$), and the Gagliardo-Nirenberg inequality \eqref{GagN2}.  Finally, the last inequality follows from the fact that $u_\epsilon,u\in C([0,T];V\cap H^2(\Omega)).$

By using the same technique in \cite[Lemma 3.3]{Ozsar2015}),  we have also the estimate
\begin{multline}\label{imp2}\frac{\beta^2}{2}\||u_\epsilon|^{2}u_\epsilon-|u|^{2}u\|_{L^2(\Gamma_1)}^2 \\
\le C\left(\|u_\epsilon\|_{L^\infty(\Gamma_1)}^4+\|u\|_{L^\infty(\Gamma_1)}^4\right)\|u_\epsilon-u\|_{L^2(\Gamma_1)}^2\\
\le  C\left(\|u_\epsilon\|_{H^1(\Gamma_1)}^4+\|u\|_{H^1(\Gamma_1)}^4\right)\|u_\epsilon-u\|_{L^2(\Gamma_1)}^2\le C\left(\|u_\epsilon\|_{H^\frac{3}{2}(\Omega)}^4+\|u\|_{H^\frac{3}{2}(\Omega)}^4\right)\|u_\epsilon-u\|_{L^2(\Gamma_1)}^2\\
\le C\left(\|u_\epsilon\|_{H^2(\Omega)}^4+\|u\|_{H^2(\Omega)}^4\right)\|u_\epsilon-u\|_{L^2(\Gamma_1)}^2\le C\|w\|_V^2,
\end{multline} where the second inequality follows from the embedding $H^1(\Gamma_1)\hookrightarrow L^\infty(\Gamma_1)$ ($\Gamma_1$ is 1-dimensional manifold), the third inequality follows from the Sobolev trace theorem, and the last inequality follows from the fact that $u,u_\epsilon\in C([0,T];H^2(\Omega))$ and the trace theorem.

Combining \eqref{impmain}-\eqref{imp2}, it follows that

\begin{multline}\label{impgood}
\frac{1}{2}\|\nabla w(t)\|_{L^2(\Omega)}^2\le \frac{1}{2}\|\nabla w_0^\epsilon\|_{L^2(\Omega)}^2+C\int_0^T\|w(t)\|_V^2dt\\
+\lambda\|\Delta u_\epsilon\|_{L^2(\Omega)}\|\Delta w\|_{L^2(\Omega)}+\kappa\|u_\epsilon\|_{L^6(\Omega)}^{3}\|\Delta w\|_{L^2(\Omega)}ds.
\end{multline}

In the above inequality $\|u_\epsilon\|_{L^6(\Omega)}$, $\|\Delta w\|_{L^2(\Omega)}$, and $\|\Delta u_\epsilon\|_{L^2(\Omega)}$ are bounded by a constant since $H^1(\Omega)\hookrightarrow L^6(\Omega)$, and $ u_\epsilon,u\in C([0,T];H^2(\Omega))$.  Hence, \eqref{impgood} gives

\begin{equation}
  \|\nabla w(t)\|_{L^2(\Omega)}^2\le \|u_\epsilon^0-u_0\|_{L^2(\Omega)}^2+(\lambda+\kappa)C+C\int_0^T\|w(t)\|_V^2dt.
\end{equation}

Unleashing the Gronwall's inequality, we have

$$ \|u_\epsilon-u\|_V^2\le \|u_\epsilon^0-u_0\|_{L^2(\Omega)}^2+(\lambda+\kappa)C\exp(CT).$$   Hence, we have just proven that $$\|u_\epsilon-u\|_V=O(\lambda)+O(\kappa)$$ as $\epsilon=(\lambda,\kappa)\rightarrow 0.$

\section{Long-time Behaviour of Solutions}  In this section, we give the proofs of Theorems \ref{decaythm} and \eqref{decaythm2}.  We will first prove the easy case ($\gamma<0$) and then prove the more subtle case $\gamma=0.$  We know from the theory of weak solutions that for any $$u_0\in Q\equiv\{\varphi\in V\cap L^{p+1}(\Omega) \text{ such that }\gamma_0\varphi\in L^{p+1}(\Gamma_1)\},$$ there corresponds a unique global weak solution $u$ which solves \eqref{GL_equation_1}. We claim that this solution tends to zero in $H^1(\Omega)\cap L^{p+1}(\Omega)$ sense as $t\rightarrow \infty$ with an exponential rate of decay.  We will use the multiplier technique to prove this where the case $\gamma=0$ requires a special multiplier.  \\

{\bf{Case 1 ($\gamma<0$):}}\\
We consider the functional $$F(t)\equiv \frac{\alpha}{2}\|\nabla u(t)\|_{L^2(\Omega)}^2+\frac{\beta}{p+1}\|u(t)\|_{L^{p+1}(\Omega)}^{p+1}.$$  Note that $F(t)\le E(t)$ given in \ref{EtDef2}.  Moreover, we have from \eqref{energyEst2} that
\begin{equation}\label{energyEst3}
 F(t)\le  E(0)+\gamma \int_0^tF(s)ds.
\end{equation}
Employing Gronwall's inequality:
\begin{equation}\label{energyEst4}
F(t)\le E_0e^{-|\gamma| t}
\end{equation} for $t\ge 0,$ where

\begin{multline}E_0=E(0)=\frac{\alpha}{2} \|\nabla u_0\|_{L^2(\Omega)}^2 +\frac{\beta}{p+1}\|u_0\|_{L^{p+1}(\Omega)}^{p+1} -\frac{\alpha\gamma}{2}\|u_0\|_{L^2(\Gamma_1)}^2\\+\frac{1}{p+1}(\alpha\kappa+\beta\lambda)\|u_0\|_{L^{p+1}(\Gamma_1)}^{p+1}\ge 0.\end{multline}

 This proves that solutions decay in $H^1(\Omega)\cap L^{p+1}(\Omega)$ sense to zero at an exponential rate.\\

{\bf{Case 2 ($\gamma=0$):}}\\  This case is more subtle and requires some control theoretic tools.  We know that solutions of the CGLE with periodic or homogeneous Dirichlet/Neumann boundary conditions do not have to decay if $\gamma\ge 0.$  We will prove that the dynamic boundary condition plays the role of a damping term and actually stabilizes the system from the boundary in the absence of an internal damping mechanism.  Therefore, the boundary condition $\displaystyle \frac{\partial u}{\partial \nu}\bigg|_{\Gamma_1}=-u_t$ can be considered a stabilizing boundary feedback within this context.

Note that Lemma \ref{H1global} tells us that when $\gamma=0$, the energy is non-increasing.  In order to prove stabilization, we will use an integral inequality given in the following lemma.

\begin{lem}\label{intinlem}\cite[Theorem 8.1]{Komornik} Let $F:\mathbb{R}_+\rightarrow\mathbb{R}_+$ be a non-increasing function and assume that there exists a constant $C>0$ such that
\begin{equation}\label{integineq}
  \int_t^\infty F(s)ds\le CF(t) \text{ for all } t\ge 0.
\end{equation}  Then, $$F(t)\le F(0)e^{1-\frac{t}{C}}  \text{ for all } t\ge 0.$$

\end{lem}

In order to obtain an integral inequality in the form of \eqref{integineq}, we calculate \begin{equation}\label{spemultiplier}\frac{d}{dt}(u,q\cdot\nabla u)_{L^2(\Omega)},\end{equation} which is by now a classical multiplier used in the control theory of PDEs and well-posedness theory of nonhomogeneous initial boundary value problems. In \eqref{spemultiplier}, $q$ denotes a sufficiently smooth vector field on $\Omega$, which will be chosen in a special way later.  We have the following lemma:
\begin{lem}\label{controllemma} Let $u$ be a global weak solution of \eqref{GL_equation_1} and $q\in \left[C^2(\bar{\Omega})\right]^n$ be a real vector field over $\Omega$.  Then, the following identity holds true:

\begin{multline}
  \frac{d}{dt}(u,q\cdot\nabla u)_{L^2(\Omega)}-\int_{\Gamma}(q\cdot \nu)u\bar{u}_td\Gamma-(\kappa-i\beta)\int_\Omega div(q)|u|^{p+1}dx\\
(\lambda-i\alpha)\int_{\Gamma}div(q)u\partial_\nu \bar{u} d\Gamma-\lambda\int_\Omega((\nabla(div(q)\cdot\nabla \bar{u})u+div(q)|\nabla u|^2))dx\\
  =2i\lambda Im(\Delta u,q\cdot\nabla u)-2i\kappa Im(|u|^{p-1}u,q\cdot\nabla u)+2i\alpha\int_\Gamma (\partial_{\nu}u )(q\cdot \nabla \bar{u})d\Gamma\\
  -2i\alpha\text{Re}\sum_{m,j=1}^n((\partial_{x_m}q_j)u_{x_m},u_{x_j})-i\alpha\int_\Gamma(q\cdot \nu)|\nabla u|^2d\Gamma\\
 -\frac{2}{p+1}i\beta\int_\Gamma (q\cdot n)|u|^{p+1}d\Gamma+\frac{2}{p+1}i\beta\int_\Omega div(q)|u|^{p+1}dx.
\end{multline}
\end{lem}

\begin{proof} Since the proof can be made by slightly modifying the proof of  \cite[Lemma 2.1]{Gal_GL} and is based on integration by parts and tedious calculations. It is omitted here.
\end{proof}

Let $q=x-x_0$.  Then, by the given geometric assumption on the boundary of $\Omega$ in Theorem \ref{decaythm2} we have $q\cdot \nu >0$ on $\Gamma_1$ and $q\cdot \nu\le 0$ on $\Gamma_0$. We can also simply calculate $div(q)=N.$ It is important to notice that since $u|_{\Gamma_0}\equiv 0$, one has $\nabla u|_{\Gamma_0}=(\partial_\nu u)\nu.$ Now, using these facts and Lemma \ref{controllemma}, it follows that
\begin{multline}\label{dec-est}
 2\alpha\int_t^T\|\nabla u\|_{L^2(\Omega)}^2dt+\frac{N\beta(p-1)}{p+1}\int_t^T\|u\|_{L^{p+1}(\Omega)}^{p+1}dt\le C(\|u(T)\|_V^2+\|u(t)\|_V^2)+\epsilon\int_t^T\|u\|_V^2dt\\
+C_\epsilon\int_t^T\left(\|u_t\|_{L^2(\Gamma_1)}^2dt+\|\Delta u\|_{L^2(\Omega)}^2+\|u\|_{L^{2p}(\Omega)}^{2p}\right).
\end{multline}

But from Lemma \ref{H1global}, we know that

$$\int_t^T\left(\|u_t\|_{L^2(\Gamma_1)}^2dt+\|\Delta u\|_{L^2(\Omega)}^2+\|u\|_{L^{2p}(\Omega)}^{2p}\right)\le C(F(t)-F(T)).$$

Combining the last inequality with  \eqref{dec-est}, it follows that

$$\int_t^TF(t)dt\le CF(T)+C(F(t)-F(T))\le CF(t),$$ and hence $$\int_t^\infty F(t)dt\le CF(t)$$ for $t\ge 0.$ Now, by using Lemma \ref{intinlem}, exponential decay follows, and the proof of Theorem \ref{decaythm2} is completed.


\begin{thebibliography}{99}

\bibitem{AamSmyKrs2005}
Ole~Morten Aamo, Andrey Smyshlyaev, and Miroslav Krsti{\'c}, {Boundary
  control of the linearized {G}inzburg-{L}andau model of vortex shedding}, SIAM
  J. Control Optim. \textbf{43} (2005), no.~6, 1953--1971

\bibitem{AamSmyKrs2007}
Ole~Morten Aamo, Andrey Smyshlyaev, Miroslav Krsti{\'c}, and Bjarne~A. Foss,
  {Output feedback boundary control of a {G}inzburg-{L}andau model of
  vortex shedding}, IEEE Trans. Automat. Control \textbf{52} (2007), no.~4,
  742--748.

\bibitem{IL2002} Aranson, I.S.; Kramer, L. {The world of the complex Ginzburg-Landau equation}, Rev. Modern Phys. 74 (2002), no.1, 99--143.

\bibitem{Audiard1}  Audiard, Corentin. On the non-homogeneous boundary value problem for Schrödinger equations. Discrete Contin. Dyn. Syst. 33 (2013), no. 9, 3861--3884.
\bibitem{Audiard2}  Audiard, Corentin. On the boundary value problem for the Schrödinger equation: compatibility conditions and global existence. Anal. PDE 8 (2015), no. 5, 1113--1143.

	\bibitem{Batal} Batal, A.; Özsar\i, T.  {Nonlinear Schrödinger equations on the half-line with nonlinear boundary conditions.} Electron. J. Differential Equations 2016, Paper No. 222, 20 pp.
	
\bibitem{Bechouche}Bechouche, P.; J\"{u}ngel A. {Inviscid Limits of the Complex Ginzburg–Landau Equation.}  Comm. Math. Phys. 214 (2000), no. 1, 201--226.


	\bibitem{Brezisbook} Br\'ezis H.  {Functional analysis, Sobolev spaces and partial differential equations.}
Universitext. Springer, New York, 2011.
	
		\bibitem{Brezis-Gallouet}Br\'ezis, H.; Gallouet, T. {Nonlinear Schrödinger evolution equations.} Nonlinear Anal. 4 (1980), no. 4, 677–681.
\bibitem{cazenave}	Cazenave, T.; Dickstein, F.; Weissler, F.B. {Finite-time blow-up for a complex Ginzburg-Landau equation.}  SIAM J. Math. Anal. 45 (2013), no. 1, 244--266.
		
\bibitem{Clement}Clément, Philippe; Okazawa, Noboru; Sobajima, Motohiro; Yokota, Tomomi
A simple approach to the Cauchy problem for complex Ginzburg-Landau equations by compactness methods.
J. Differential Equations 253 (2012), no. 4, 1250--1263.
			
	\bibitem{Dynamic_BC} Cavalcanti, Marcelo M.; Corrêa, Wellington J.; Lasiecka, Irena; Lefler, Christopher. Well-posedness and uniform stability for nonlinear Schrödinger equations with dynamic/Wentzell boundary conditions.  Indiana Univ. Math. J. 65 (2016), no. 5, 1445--1502.
	
	\bibitem{Cavalcanti1}Cavalcanti, Marcelo M.; Lasiecka, Irena; Toundykov, Daniel.
Geometrically constrained stabilization of wave equations with Wentzell boundary conditions.
Appl. Anal. 91 (2012), no. 8, 1427--1452.
	
	\bibitem{doering}Doering, Charles R.; Gibbon, John D.; Levermore, C. David. Weak and strong solutions of the complex Ginzburg-Landau equation. Phys. D 71 (1994), no. 3, 285–318.

    \bibitem{erdogan} Erdoğan, M. B.; Tzirakis, N. Regularity properties of the cubic nonlinear Schrödinger equation on the half line. J. Funct. Anal. 271 (2016), no. 9, 2539--2568. 	
    \bibitem{fokas} Fokas, Athanassios S.; Himonas, A. Alexandrou; Mantzavinos, Dionyssios. The nonlinear Schrödinger equation on the half-line. Trans. Amer. Math. Soc. 369 (2017), no. 1, 681--709.

	\bibitem{Gal_GL}  Gao, Hongjun; Bu, Charles. Dirichlet inhomogeneous boundary value problem for the n+1 complex Ginzburg-Landau equation. J. Differential Equations 198 (2004), no. 1, 176--195.
	
	\bibitem{Gao_GL_1}Gao, Hongjun; Gu, Xiaohua; Bu, Charles. A Neumann boundary value problem for a generalized Ginzburg-Landau equation. Appl. Math. Comput. 134 (2003), no. 2-3, 553--560.
	
\bibitem{Ginibre1}	Ginibre, J.; Velo, G. The Cauchy problem in local spaces for the complex Ginzburg-Landau equation. I. Compactness methods. Phys. D 95 (1996), no. 3-4, 191--228.
	
	\bibitem{Ginibre2} Ginibre, J.; Velo, G. The Cauchy problem in local spaces for the complex Ginzburg-Landau equation. II. Contraction methods. Comm. Math. Phys. 187 (1997), no. 1, 45--79.

	\bibitem{Goldstein}   Favini, Angelo; Goldstein, Gisèle Ruiz; Goldstein, Jerome A.; Romanelli, Silvia. The heat equation with generalized Wentzell boundary condition.  J. Evol. Equ. 2 (2002), no. 1, 1--19
	
	\bibitem{Goldstein1}Favini, Angelo; Goldstein, Giséle Ruiz; Goldstein, Jerome A.; Romanelli, Silvia. $C_0$-semigroups generated by second order differential operators with general Wentzell boundary conditions. Proc. Amer. Math. Soc. 128 (2000), no. 7, 1981--1989.
	
\bibitem{Goldstein2} Favini, Angelo; Goldstein, Gisèle Ruiz; Goldstein, Jerome A.; Romanelli, Silvia The heat equation with generalized Wentzell boundary condition.
J. Evol. Equ. 2 (2002), no. 1, 1–19.

\bibitem{holmer} Holmer, Justin. The initial-boundary-value problem for the 1D nonlinear Schrödinger equation on the half-line. Differential Integral Equations 18 (2005), no. 6, 647--668.

\bibitem{elena} Kaikina, Elena I. Inhomogeneous Neumann initial-boundary value problem for the nonlinear Schrödinger equation. J. Differential Equations 255 (2013), no. 10, 3338--3356.

	\bibitem{Kalantarov} Kalantarov, Varga K.; Özsarı, Türker. Qualitative properties of solutions for nonlinear Schrödinger equations with nonlinear boundary conditions on the half-line. J. Math. Phys. 57 (2016), no. 2, 021511, 14 pp.

	\bibitem{Komornik} Komornik, V. Exact controllability and stabilization.  The multiplier method. RAM: Research in Applied Mathematics. Masson, Paris; John Wiley \& Sons, Ltd., Chichester, 1994.

	\bibitem{Lions-Magenes}  Lions, J.-L.; Magenes, E. Problèmes aux limites non homogènes et applications. Vol. 1.  Travaux et Recherches Mathématiques, No. 17 Dunod, Paris 1968
	
	\bibitem{Lasiecka0}   Lasiecka, I.; Tataru, D. Uniform boundary stabilization of semilinear wave equations with nonlinear boundary damping. Differential Integral Equations 6 (1993), no. 3, 507--533.
	
	
	
	\bibitem{Lasiecka1}   Lasiecka, I.; Triggiani, R. Optimal regularity, exact controllability and uniform stabilization of Schrödinger equations with Dirichlet control. Differential Integral Equations 5 (1992), no. 3, 521--535.
	
	\bibitem{Lasiecka2} Lasiecka, Irena; Triggiani, Roberto. Well-posedness and sharp uniform decay rates at the $L^2(\Omega)$-level of the Schrödinger equation with nonlinear boundary dissipation. J. Evol. Equ. 6 (2006), no. 3, 485–537.
	
	 \bibitem{Lions} Lions, J.-L. Quelques méthodes de résolution des problèmes aux limites non linéaires. Dunod; Gauthier-Villars, Paris 1969

	\bibitem{Nader}Masmoudi, N.; Zaag, H. {Blow-up profile for the complex Ginzburg-Landau equation.} J. Funct. Anal. 255 (2008), no. 7, 1613–1666.

	\bibitem{okazawa}  Okazawa, Noboru. Smoothing effect and strong L2-wellposedness in the complex Ginzburg-Landau equation. Differential equations: inverse and direct problems, 265--288, Lect. Notes Pure Appl. Math., 251, Chapman \& Hall/CRC, Boca Raton, FL, 2006.

\bibitem{Ogawa}  Ogawa, Takayoshi; Yokota, Tomomi. Uniqueness and inviscid limits of solutions for the complex Ginzburg-Landau equation in a two-dimensional domain. Comm. Math. Phys. 245 (2004), no. 1, 105--121.
	
	\bibitem{okazawa_sectorial} Okazawa, Noboru.  Sectorialness of second order elliptic operators in divergence form. Proc. Amer. Math. Soc. 113 (1991), no. 3, 701–706.
	
	\bibitem{okazawa2002monotonicity}  Okazawa, Noboru; Yokota, Tomomi. Monotonicity method applied to the complex Ginzburg-Landau and related equations. J. Math. Anal. Appl. 267 (2002), no. 1, 247--263.



	\bibitem{Ozsari}  Özsarı, Türker; Kalantarov, Varga K.; Lasiecka, Irena. Uniform decay rates for the energy of weakly damped defocusing semilinear Schrödinger equations with inhomogeneous Dirichlet boundary control. J. Differential Equations 251 (2011), no. 7, 1841--1863.
	

	
	\bibitem{Ozsar2}  Özsarı, Türker. Weakly-damped focusing nonlinear Schrödinger equations with Dirichlet control. J. Math. Anal. Appl. 389 (2012), no. 1, 84--97.

	
	
	
	\bibitem{Ozsari3}  Özsarı, Türker. Global existence and open loop exponential stabilization of weak solutions for nonlinear Schrödinger equations with localized external Neumann manipulation. Nonlinear Anal. 80 (2013), 179--193.
	
	\bibitem{Ozsar2015}  Özsarı, Türker. Well-posedness for nonlinear Schrödinger equations with boundary forces in low dimensions by Strichartz estimates. J. Math. Anal. Appl. 424 (2015), no. 1, 487--508.
	
	\bibitem{Pazy} Pazy, A. Semigroups of linear operators and applications to partial differential equations. Applied Mathematical Sciences, 44. Springer-Verlag, New York, 1983.
	
	\bibitem{Rosier} Rosier, Lionel; Zhang, Bing-Yu. Null controllability of the complex Ginzburg-Landau equation. Ann. Inst. H. Poincaré Anal. Non Linéaire 26 (2009), no. 2, 649--673.
	
		\bibitem{Show}   Showalter, R. E. Monotone operators in Banach space and nonlinear partial differential equations. Mathematical Surveys and Monographs, 49. American Mathematical Society, Providence, RI, 1997.
	
	\bibitem{Shimo} Shimotsuma, Daisuke; Yokota, Tomomi; Yoshii, Kentarou. Existence and decay estimates of solutions to complex Ginzburg-Landau type equations. J. Differential Equations 260 (2016), no. 3, 3119--3149.
	
	
	\bibitem{shimotsuma}	 Shimotsuma, Daisuke; Yokota, Tomomi; Yoshii, Kentarou. Cauchy problem for the complex Ginzburg-Landau type equation with $L^p$-initial data. Math. Bohem. 139 (2014), no. 2, 353--361.
	
	\bibitem{Strauss}  Strauss, Walter; Bu, Charles.
An inhomogeneous boundary value problem for nonlinear Schrödinger equations.
J. Differential Equations 173 (2001), no. 1, 79--91.
	
\bibitem{tang}  Tang, Qi; Wang, S. Time dependent Ginzburg-Landau equations of superconductivity. Phys. D 88 (1995), no. 3-4, 139--166.
	
	\bibitem{Wentzell}  Ventcel', A. D. On boundary conditions for multi-dimensional diffusion processes. Theor. Probability Appl. 4 1959 164--177.

\bibitem{Wu} Wu, Jiahong. {The inviscid limit of the complex Ginzburg-Landau equation.} J. Differential Equations 142 (1998), no. 2, 413--433.
	

\end{thebibliography}
 \end{document}